\documentclass[dvipdfmx, 12pt] {article}

\usepackage{tikz}
\usepackage{pgf}
\usetikzlibrary{arrows}
\usepackage{here}

\usepackage{amsmath, amsfonts, amssymb, color, ascmac}
\usepackage{color}

\newcommand{\beq}{\begin{eqnarray*}}
\newcommand{\eeq}{\end{eqnarray*}}

\newtheorem{theorem}{Theorem}[section]
\newtheorem{lemma}[theorem]{Lemma}
\newtheorem{corollary}[theorem]{Corollary}

\newtheorem{remark}[theorem]{Remark}

\newsavebox{\toy}
\savebox{\toy}{\framebox[0.65em]{\rule{0cm}{1ex}}}
\newcommand{\QED}{\usebox{\toy}}
\def\nlni{\par\ifvmode\removelastskip\fi\vskip\baselineskip\noindent}
\newenvironment{proof}{\nlni\begingroup\it Proof.\rm}{
\endgroup\vskip\baselineskip}

\makeatletter
\providecommand*{\shuffle}{%
  \mathbin{\mathpalette\shuffle@{}}%
}
\newcommand*{\shuffle@}[2]{%
  \sbox0{$#1\vcenter{}$}%
  \kern .15\ht0 
  \rlap{\vrule height .25\ht0 depth 0pt width 2.5\ht0}%
  \raise.1\ht0\hbox to 2.5\ht0{%
    \vrule height 1.75\ht0 depth -.1\ht0 width .17\ht0 %
    \hfill
    \vrule height 1.75\ht0 depth -.1\ht0 width .17\ht0 %
    \hfill
    \vrule height 1.75\ht0 depth -.1\ht0 width .17\ht0 %
  }%
  \kern .15\ht0 
}
\makeatother

\begin{document}

\setlength{\baselineskip}{15pt}
\title{
Top to random shuffles on colored permutations
}
\author{
Fumihiko Nakano
\thanks{
Mathematical Institute, 
Tohoku University, 
Sendai 980-8578, Japan
e-mail:
fumihiko.nakano.e4@tohoku.ac.jp
}
\and 
Taizo Sadahiro
\thanks{Department of Computer Science, 
Tsuda University, 
2-1-1, Tsuda, Kodaira City, 187-8577, Tokyo, Japan.
e-mail : sadahiro@tsuda.ac.jp}
\and
Tetsuya Sakurai
\thanks{
Mathematical Institute,
Tohoku University,
Sendai 980-8578, Japan
e-mail : 
tetsuya.sakurai.t8@dc.tohoku.ac.jp
}
}
\maketitle
\begin{abstract}
A deck of 
$n$ 
cards are shuffled by repeatedly taking off the top card,  flipping it  with probability
$1/2$, 
and inserting it back into the deck at a random position. 
This process 
can be considered as a Markov chain on the group 
$B_n$ 
of signed permutations.
We show 
that the eigenvalues of the transition probability matrix are 
$0,1/n,2/n,\ldots,(n-1)/n,1$ 
and the multiplicity of the eigenvalue 
$i/n$ 
is equal to the number of the {\em signed} permutation having exactly 
$i$ 
fixed points.
We show 
the similar results hold also for the colored permutations. 
Further, 
we show 
that the mixing time of this Markov chain is 
$n\log n$ 
and exhibits cut off, same as the ordinary 'top to random' shuffles without flipping the cards.
The cut off 
is also analyzed by using the asymptotic formula of the Stirling numbers of the second kind.
\end{abstract}


\section{Introduction}
The {\em top to random shuffle} 
of cards, 
which is a Markov chain on the symmetric group, has long been studied \cite{aldousDiaconis1986,DiaconisFillPitman1992}.
By modifying 
the arguments in 
$\cite{DiaconisFillPitman1992}$ 
and 
$\cite{GarsiaLecNote}$,
this paper studies the generalized top to random shuffling defined on the colored permutation group.
For  example, 
when the number of the 'colors' is two, the
Markov chain can be described as follows:
We take 
the top card of the deck of 
$n$ 
cards and before inserting it back into the deck at
the random position, we flip the card with probability $1/2$.
After repeating 
this procedure, we have a random configuration of the cards which can be regarded as an element of the hyperoctahedral
group 
$B_n$.
This generalization 
is similar to those of the riffle shuffle
\cite{bergeron1992orthogonal, NakanoSadahiroRiffle}.
Our main aim 
in this paper is to show a closed formula describing the probability distribution after
shuffling 
$k$ times 
in terms of (generalized) Stirling numbers, 
the explicit form of the eigenvalues of the transition probability matrix, and the mixing time and cut off of the Markov chain.
By a {\em colored permutation group}, 
we mean the wreath product of a cyclic group and a symmetric group.
Throughout, 
the symmetric group of degree 
$n$ 
is denoted by $\mathfrak S_n$,
and the cyclic group of order 
$p$ 
is denoted by $C_p$.
For positive integers 
$p$ and $n$,
the wreath product 
$C_p\wr {\mathfrak S}_n$ 
is denoted by $G_{n,p}$.
That is, 
the colored permutation group 
$G_{n,p}$ 
is defined by
\[
 G_{n,p} = \left\{
 (s,\sigma)\,|\,
 s=(s_1,\ldots,s_n)\in C_p^n,
 \sigma\in{\mathfrak S}_n
 \right\}
\]
equipped with the following multiplication rule,
\[
 (t,\tau)(s,\sigma)
 =
 (\sigma t+s, \tau\sigma)
\]
for $(t,\tau)$ and $(s,\sigma) \in G_{n,p}$, where
\[
 \sigma t = (t_{\sigma(1)},t_{\sigma(2)},\ldots,t_{\sigma(n)}).
\]
For example, let $s=(0,\ldots,0,\stackrel{\stackrel{i}{\vee}}{k},0,\ldots,0)$ 
and
$\sigma=(1,2,\ldots,i)$ 
be a cyclic permutation in 
${\mathfrak S}_n$.
Then,
\[
 (t,\tau)(s,\sigma) =
 ((t_2,t_3,\ldots,t_i,t_1+k,t_{i+1},\ldots,t_{n}),(\tau(2),\tau(3),\ldots,\tau(i),\tau(1),\tau(i+1),\ldots,\tau(n))).
\]
We interpret 
the multiplication by this special element 
$S_{i,k}=(s,\sigma)$ 
as follows:
We have 
a deck of $n$ 
cards each numbered 
$\tau(1),\tau(2),\ldots,\tau(n)$ 
from the top to the bottom.
We take 
the top card and change the {\em color} of
the card to 
$t_1+k$
and insert it into the deck at the $i$ th place from the top.
To introduce 
the shuffle we regard 
$(s,\sigma)\in G_{n,p}$ 
as a sequence of the pairs
$(s_i,\sigma(i))\in C_p\times [n]$, 
so that
$(s,\sigma)$ 
is a word over the alphabet $C_p\times [n]$.
Especially for 
$p\leq 3$, 
we express
$(0,k)\in C_p\times[n]$  
simply by 
$k$,
$(1,k)\in C_p\times[n]$ 
by 
$\bar{k}$,
and 
$(2,k)$ 
by 
$\bar{\bar{k}}$.
For example, 
by using this notation,
$((0,1,0,2), 4123)\in G_{4,3}$ 
can be simply expressed as 
$4\bar{1}2\bar{\bar{3}}$.
Thus elements of 
$G_{n,p}$ 
can be considered as words over the alphabet 
$C_p\times [n]$, 
on which we can use the {\em shuffle operator} 
$\shuffle$.
Here 
the shuffle operator 
$\shuffle$
is defined as follows.
Let 
$\epsilon$
be the empty word, 
$u, v$
be any words, and let 
$a, b$
be the words of length 
$1$.
Then  
$\shuffle$
between two words is defined inductively by the following equations.
\beq
u \shuffle \epsilon
&=&
\epsilon \shuffle u
:=
u
\\
ua \shuffle vb
&:=&
( u\shuffle vb) a 
+
(ua \shuffle v) b.
\eeq
Define the word 
$W_{k,n}$ 
by
\[
 W_{k,n} 
 := 
 (0,k+1) (0,k+2) \cdots (0,n) \in (C_p\times[n])^{n-k}.
\]
Then 
an element 
${\mathbf B}_k$ 
of the group
algebra 
${\mathbb Q}G_{n,p}$  
is defined by
\[
 {\mathbf B}_k 
 :=
 \begin{cases}
 id & (k=0) \\
 \sum_{\alpha \in G_{k,p}} \alpha\shuffle W_{k,n}, 
 & (1 \le k \le n-1) \\
 \sum_{\alpha \in G_{n, p}} \alpha & (k=n)
 \end{cases}
\]
${\bf B}_k$
for 
$k \ge 2$ 
can be regarded as generalized top to random shuffle, 
which corresponds to taking off top 
$k$
cards, flipping them into any colors, and inserting back into random positions.
In particular, we have
 \[
  {\mathbf B}_1 = (0,1)\shuffle W_{1,n} + (1,1)\shuffle W_{1,n} 
  + \cdots + (p-1,1)\shuffle W_{1,n}.
 \]
For example, 
when 
$n=3$ 
and 
$p=2$, 
we have,
\[
{\mathbf B}_1 
= 
1\shuffle 23 + \bar{1}\shuffle 23
= 123 + 213 + 231 + \bar{1}23 + 2\bar{1}3 + 23\bar{1}.
 \]
Therefore 
$\dfrac{1}{np}{\bf B}_1$ 
can be regarded as a
probability distribution over 
$G_{n,p}$,
which we call the {\em top to random shuffle} over 
$G_{n,p}$.
When 
$p=1$, 
the powers of 
$\left(\dfrac{1}{np}{\bf B}_1\right)^k$
exhibit very interesting properties and have been studied extensively
\cite{DiaconisFillPitman1992}.
The main purpose 
of this paper is to consider the case for general 
$p$ : 
(i)
to give a precise description
of the distribution of the eigenvalues of the left regular
representation of 
${\bf B}_1$, 
and
(ii) 
to derive a sharp estimate on the distance between the distribution of 
$\left(\dfrac{1}{np}{\bf B}_1\right)^k$
and the stationary distribution, and show that it exhibits the cut off phenomenon. 
To state our first main result 
we need to define the {\em fixed points} of a colored permutation.
An element 
$(s,\sigma)\in G_{n,p}$ 
has a fixed point at 
$i$
if 
$s_i=0$ 
and 
$\sigma(i)=i$.
For example 
$4\bar{1}2\bar{\bar{3}} \in G_{4,3}$ 
has no fixed point and
$1\bar{2}35\bar{\bar{4}}$ 
has two fixed points at 
$1$ 
and 
$3$.
A {\em derangement} in 
$G_{n,p}$ 
is a colored permutation having no fixed points.
We denote 
the number of derangement in 
$G_{n,p}$ 
by 
$D_{n,p}$,
which is expressed by a closed form given later.
Therefore, the number of colored permutations in
$G_{n,p}$ 
having exactly 
$i$ 
fixed points is
equal to 
${n\choose i}D_{n-i,p}$.
 \begin{theorem}
 \label{thm:main}
 Let 
 $L:G_{n,p}\rightarrow {\rm GL}(L^2(G_{n,p}))$ 
 be the left regular representation of 
 $G_{n,p}$.
 Then 
 the eigenvalues of 
 $L({\mathbf B}_1)$ 
 are
 $0,p,2p,\cdots,np$.
 The multiplicity 
 of the eigenvalue 
 $ip$ 
 $(i=0,1, \cdots, n)$
 is equal to
 the number of colored permutations having exactly 
 $i$
 fixed points.
\end{theorem}
{\bf Remark }\\
(1)
Let 
$P_{n, p}$
be the transition probability matrix 
of the Markov chain generated by the top to random shuffle.
Then 
the eigenvalues of 
$P_{n, p}$
are given by 
$(ip) / (np) =i / n$, 
$i=0, 1, \cdots, n$.
\\
(2)
If 
$p=1$, 
$P_{n, 1}$
does not have 
$(n-1)/n$
as an eigenvalue because of 
$D_{n - (n-1), 1} = 0$.
It is not the case for 
$p \ge 2$, 
since 
$D_{n - (n-1), \, p} \ne 0$
for 
$p \ge 2$. \\
{\bf Example}\\
When 
$n=3$ 
and 
$p=2$, 
we have 
$48$ 
elements in 
$G_{n,p}$.
The left 
regular representation of 
${\mathbf B}_1$ 
is
 {\tiny
 \[\arraycolsep=1.4pt\def\arraystretch{1.0}
 \left(\begin{array}{rrrrrrrrrrrrrrrrrrrrrrrrrrrrrrrrrrrrrrrrrrrrrrrr}
1 & 0 & 0 & 0 & 1 & 0 & 0 & 0 & 0 & 0 & 0 & 0 & 1 & 0 & 0 & 0 & 0 & 0 & 0 & 0 & 0 & 0 & 0 & 0 & 0 & 0 & 1 & 0 & 1 & 0 & 0 & 0 & 0 & 0 & 0 & 0 & 0 & 0 & 1 & 0 & 0 & 0 & 0 & 0 & 0 & 0 & 0 & 0 \\
0 & 1 & 0 & 0 & 0 & 1 & 0 & 0 & 0 & 0 & 0 & 0 & 0 & 0 & 1 & 0 & 0 & 0 & 0 & 0 & 0 & 0 & 0 & 0 & 0 & 0 & 0 & 1 & 0 & 1 & 0 & 0 & 0 & 0 & 0 & 0 & 1 & 0 & 0 & 0 & 0 & 0 & 0 & 0 & 0 & 0 & 0 & 0 \\
0 & 0 & 1 & 0 & 0 & 0 & 1 & 0 & 0 & 0 & 0 & 0 & 0 & 1 & 0 & 0 & 0 & 0 & 0 & 0 & 0 & 0 & 0 & 0 & 1 & 0 & 0 & 0 & 0 & 0 & 1 & 0 & 0 & 0 & 0 & 0 & 0 & 0 & 0 & 1 & 0 & 0 & 0 & 0 & 0 & 0 & 0 & 0 \\
0 & 0 & 0 & 1 & 0 & 0 & 0 & 1 & 0 & 0 & 0 & 0 & 0 & 0 & 0 & 1 & 0 & 0 & 0 & 0 & 0 & 0 & 0 & 0 & 0 & 1 & 0 & 0 & 0 & 0 & 0 & 1 & 0 & 0 & 0 & 0 & 0 & 1 & 0 & 0 & 0 & 0 & 0 & 0 & 0 & 0 & 0 & 0 \\
0 & 0 & 0 & 0 & 1 & 0 & 0 & 0 & 1 & 0 & 0 & 0 & 0 & 0 & 0 & 0 & 0 & 0 & 0 & 0 & 1 & 0 & 0 & 0 & 0 & 0 & 0 & 0 & 1 & 0 & 0 & 0 & 0 & 1 & 0 & 0 & 0 & 0 & 0 & 0 & 0 & 0 & 0 & 0 & 0 & 1 & 0 & 0 \\
0 & 0 & 0 & 0 & 0 & 1 & 0 & 0 & 0 & 1 & 0 & 0 & 0 & 0 & 0 & 0 & 0 & 0 & 0 & 0 & 0 & 0 & 1 & 0 & 0 & 0 & 0 & 0 & 0 & 1 & 0 & 0 & 1 & 0 & 0 & 0 & 0 & 0 & 0 & 0 & 0 & 0 & 0 & 0 & 0 & 0 & 0 & 1 \\
0 & 0 & 0 & 0 & 0 & 0 & 1 & 0 & 0 & 0 & 1 & 0 & 0 & 0 & 0 & 0 & 0 & 0 & 0 & 0 & 0 & 1 & 0 & 0 & 0 & 0 & 0 & 0 & 0 & 0 & 1 & 0 & 0 & 0 & 0 & 1 & 0 & 0 & 0 & 0 & 0 & 0 & 0 & 0 & 1 & 0 & 0 & 0 \\
0 & 0 & 0 & 0 & 0 & 0 & 0 & 1 & 0 & 0 & 0 & 1 & 0 & 0 & 0 & 0 & 0 & 0 & 0 & 0 & 0 & 0 & 0 & 1 & 0 & 0 & 0 & 0 & 0 & 0 & 0 & 1 & 0 & 0 & 1 & 0 & 0 & 0 & 0 & 0 & 0 & 0 & 0 & 0 & 0 & 0 & 1 & 0 \\
1 & 0 & 0 & 0 & 0 & 0 & 0 & 0 & 1 & 0 & 0 & 0 & 0 & 0 & 0 & 0 & 1 & 0 & 0 & 0 & 0 & 0 & 0 & 0 & 0 & 0 & 1 & 0 & 0 & 0 & 0 & 0 & 0 & 1 & 0 & 0 & 0 & 0 & 0 & 0 & 1 & 0 & 0 & 0 & 0 & 0 & 0 & 0 \\
0 & 1 & 0 & 0 & 0 & 0 & 0 & 0 & 0 & 1 & 0 & 0 & 0 & 0 & 0 & 0 & 0 & 0 & 1 & 0 & 0 & 0 & 0 & 0 & 0 & 0 & 0 & 1 & 0 & 0 & 0 & 0 & 1 & 0 & 0 & 0 & 0 & 0 & 0 & 0 & 0 & 0 & 1 & 0 & 0 & 0 & 0 & 0 \\
0 & 0 & 1 & 0 & 0 & 0 & 0 & 0 & 0 & 0 & 1 & 0 & 0 & 0 & 0 & 0 & 0 & 1 & 0 & 0 & 0 & 0 & 0 & 0 & 1 & 0 & 0 & 0 & 0 & 0 & 0 & 0 & 0 & 0 & 0 & 1 & 0 & 0 & 0 & 0 & 0 & 1 & 0 & 0 & 0 & 0 & 0 & 0 \\
0 & 0 & 0 & 1 & 0 & 0 & 0 & 0 & 0 & 0 & 0 & 1 & 0 & 0 & 0 & 0 & 0 & 0 & 0 & 1 & 0 & 0 & 0 & 0 & 0 & 1 & 0 & 0 & 0 & 0 & 0 & 0 & 0 & 0 & 1 & 0 & 0 & 0 & 0 & 0 & 0 & 0 & 0 & 1 & 0 & 0 & 0 & 0 \\
1 & 0 & 0 & 0 & 0 & 0 & 0 & 0 & 0 & 0 & 0 & 0 & 1 & 0 & 0 & 0 & 1 & 0 & 0 & 0 & 0 & 0 & 0 & 0 & 0 & 0 & 1 & 0 & 0 & 0 & 0 & 0 & 0 & 0 & 0 & 0 & 0 & 0 & 1 & 0 & 1 & 0 & 0 & 0 & 0 & 0 & 0 & 0 \\
0 & 0 & 1 & 0 & 0 & 0 & 0 & 0 & 0 & 0 & 0 & 0 & 0 & 1 & 0 & 0 & 0 & 1 & 0 & 0 & 0 & 0 & 0 & 0 & 1 & 0 & 0 & 0 & 0 & 0 & 0 & 0 & 0 & 0 & 0 & 0 & 0 & 0 & 0 & 1 & 0 & 1 & 0 & 0 & 0 & 0 & 0 & 0 \\
0 & 1 & 0 & 0 & 0 & 0 & 0 & 0 & 0 & 0 & 0 & 0 & 0 & 0 & 1 & 0 & 0 & 0 & 1 & 0 & 0 & 0 & 0 & 0 & 0 & 0 & 0 & 1 & 0 & 0 & 0 & 0 & 0 & 0 & 0 & 0 & 1 & 0 & 0 & 0 & 0 & 0 & 1 & 0 & 0 & 0 & 0 & 0 \\
0 & 0 & 0 & 1 & 0 & 0 & 0 & 0 & 0 & 0 & 0 & 0 & 0 & 0 & 0 & 1 & 0 & 0 & 0 & 1 & 0 & 0 & 0 & 0 & 0 & 1 & 0 & 0 & 0 & 0 & 0 & 0 & 0 & 0 & 0 & 0 & 0 & 1 & 0 & 0 & 0 & 0 & 0 & 1 & 0 & 0 & 0 & 0 \\
0 & 0 & 0 & 0 & 0 & 0 & 0 & 0 & 1 & 0 & 0 & 0 & 0 & 0 & 0 & 0 & 1 & 0 & 0 & 0 & 1 & 0 & 0 & 0 & 0 & 0 & 0 & 0 & 0 & 0 & 0 & 0 & 0 & 1 & 0 & 0 & 0 & 0 & 0 & 0 & 1 & 0 & 0 & 0 & 0 & 1 & 0 & 0 \\
0 & 0 & 0 & 0 & 0 & 0 & 0 & 0 & 0 & 0 & 1 & 0 & 0 & 0 & 0 & 0 & 0 & 1 & 0 & 0 & 0 & 1 & 0 & 0 & 0 & 0 & 0 & 0 & 0 & 0 & 0 & 0 & 0 & 0 & 0 & 1 & 0 & 0 & 0 & 0 & 0 & 1 & 0 & 0 & 1 & 0 & 0 & 0 \\
0 & 0 & 0 & 0 & 0 & 0 & 0 & 0 & 0 & 1 & 0 & 0 & 0 & 0 & 0 & 0 & 0 & 0 & 1 & 0 & 0 & 0 & 1 & 0 & 0 & 0 & 0 & 0 & 0 & 0 & 0 & 0 & 1 & 0 & 0 & 0 & 0 & 0 & 0 & 0 & 0 & 0 & 1 & 0 & 0 & 0 & 0 & 1 \\
0 & 0 & 0 & 0 & 0 & 0 & 0 & 0 & 0 & 0 & 0 & 1 & 0 & 0 & 0 & 0 & 0 & 0 & 0 & 1 & 0 & 0 & 0 & 1 & 0 & 0 & 0 & 0 & 0 & 0 & 0 & 0 & 0 & 0 & 1 & 0 & 0 & 0 & 0 & 0 & 0 & 0 & 0 & 1 & 0 & 0 & 1 & 0 \\
0 & 0 & 0 & 0 & 1 & 0 & 0 & 0 & 0 & 0 & 0 & 0 & 1 & 0 & 0 & 0 & 0 & 0 & 0 & 0 & 1 & 0 & 0 & 0 & 0 & 0 & 0 & 0 & 1 & 0 & 0 & 0 & 0 & 0 & 0 & 0 & 0 & 0 & 1 & 0 & 0 & 0 & 0 & 0 & 0 & 1 & 0 & 0 \\
0 & 0 & 0 & 0 & 0 & 0 & 1 & 0 & 0 & 0 & 0 & 0 & 0 & 1 & 0 & 0 & 0 & 0 & 0 & 0 & 0 & 1 & 0 & 0 & 0 & 0 & 0 & 0 & 0 & 0 & 1 & 0 & 0 & 0 & 0 & 0 & 0 & 0 & 0 & 1 & 0 & 0 & 0 & 0 & 1 & 0 & 0 & 0 \\
0 & 0 & 0 & 0 & 0 & 1 & 0 & 0 & 0 & 0 & 0 & 0 & 0 & 0 & 1 & 0 & 0 & 0 & 0 & 0 & 0 & 0 & 1 & 0 & 0 & 0 & 0 & 0 & 0 & 1 & 0 & 0 & 0 & 0 & 0 & 0 & 1 & 0 & 0 & 0 & 0 & 0 & 0 & 0 & 0 & 0 & 0 & 1 \\
0 & 0 & 0 & 0 & 0 & 0 & 0 & 1 & 0 & 0 & 0 & 0 & 0 & 0 & 0 & 1 & 0 & 0 & 0 & 0 & 0 & 0 & 0 & 1 & 0 & 0 & 0 & 0 & 0 & 0 & 0 & 1 & 0 & 0 & 0 & 0 & 0 & 1 & 0 & 0 & 0 & 0 & 0 & 0 & 0 & 0 & 1 & 0 \\
0 & 0 & 1 & 0 & 1 & 0 & 0 & 0 & 0 & 0 & 0 & 0 & 0 & 0 & 1 & 0 & 0 & 0 & 0 & 0 & 0 & 0 & 0 & 0 & 1 & 0 & 0 & 0 & 1 & 0 & 0 & 0 & 0 & 0 & 0 & 0 & 1 & 0 & 0 & 0 & 0 & 0 & 0 & 0 & 0 & 0 & 0 & 0 \\
0 & 0 & 0 & 1 & 0 & 1 & 0 & 0 & 0 & 0 & 0 & 0 & 1 & 0 & 0 & 0 & 0 & 0 & 0 & 0 & 0 & 0 & 0 & 0 & 0 & 1 & 0 & 0 & 0 & 1 & 0 & 0 & 0 & 0 & 0 & 0 & 0 & 0 & 1 & 0 & 0 & 0 & 0 & 0 & 0 & 0 & 0 & 0 \\
1 & 0 & 0 & 0 & 0 & 0 & 1 & 0 & 0 & 0 & 0 & 0 & 0 & 0 & 0 & 1 & 0 & 0 & 0 & 0 & 0 & 0 & 0 & 0 & 0 & 0 & 1 & 0 & 0 & 0 & 1 & 0 & 0 & 0 & 0 & 0 & 0 & 1 & 0 & 0 & 0 & 0 & 0 & 0 & 0 & 0 & 0 & 0 \\
0 & 1 & 0 & 0 & 0 & 0 & 0 & 1 & 0 & 0 & 0 & 0 & 0 & 1 & 0 & 0 & 0 & 0 & 0 & 0 & 0 & 0 & 0 & 0 & 0 & 0 & 0 & 1 & 0 & 0 & 0 & 1 & 0 & 0 & 0 & 0 & 0 & 0 & 0 & 1 & 0 & 0 & 0 & 0 & 0 & 0 & 0 & 0 \\
0 & 0 & 0 & 0 & 1 & 0 & 0 & 0 & 0 & 1 & 0 & 0 & 0 & 0 & 0 & 0 & 0 & 0 & 0 & 0 & 0 & 1 & 0 & 0 & 0 & 0 & 0 & 0 & 1 & 0 & 0 & 0 & 1 & 0 & 0 & 0 & 0 & 0 & 0 & 0 & 0 & 0 & 0 & 0 & 1 & 0 & 0 & 0 \\
0 & 0 & 0 & 0 & 0 & 1 & 0 & 0 & 1 & 0 & 0 & 0 & 0 & 0 & 0 & 0 & 0 & 0 & 0 & 0 & 0 & 0 & 0 & 1 & 0 & 0 & 0 & 0 & 0 & 1 & 0 & 0 & 0 & 1 & 0 & 0 & 0 & 0 & 0 & 0 & 0 & 0 & 0 & 0 & 0 & 0 & 1 & 0 \\
0 & 0 & 0 & 0 & 0 & 0 & 1 & 0 & 0 & 0 & 0 & 1 & 0 & 0 & 0 & 0 & 0 & 0 & 0 & 0 & 1 & 0 & 0 & 0 & 0 & 0 & 0 & 0 & 0 & 0 & 1 & 0 & 0 & 0 & 1 & 0 & 0 & 0 & 0 & 0 & 0 & 0 & 0 & 0 & 0 & 1 & 0 & 0 \\
0 & 0 & 0 & 0 & 0 & 0 & 0 & 1 & 0 & 0 & 1 & 0 & 0 & 0 & 0 & 0 & 0 & 0 & 0 & 0 & 0 & 0 & 1 & 0 & 0 & 0 & 0 & 0 & 0 & 0 & 0 & 1 & 0 & 0 & 0 & 1 & 0 & 0 & 0 & 0 & 0 & 0 & 0 & 0 & 0 & 0 & 0 & 1 \\
0 & 0 & 1 & 0 & 0 & 0 & 0 & 0 & 0 & 1 & 0 & 0 & 0 & 0 & 0 & 0 & 1 & 0 & 0 & 0 & 0 & 0 & 0 & 0 & 1 & 0 & 0 & 0 & 0 & 0 & 0 & 0 & 1 & 0 & 0 & 0 & 0 & 0 & 0 & 0 & 1 & 0 & 0 & 0 & 0 & 0 & 0 & 0 \\
0 & 0 & 0 & 1 & 0 & 0 & 0 & 0 & 1 & 0 & 0 & 0 & 0 & 0 & 0 & 0 & 0 & 0 & 1 & 0 & 0 & 0 & 0 & 0 & 0 & 1 & 0 & 0 & 0 & 0 & 0 & 0 & 0 & 1 & 0 & 0 & 0 & 0 & 0 & 0 & 0 & 0 & 1 & 0 & 0 & 0 & 0 & 0 \\
1 & 0 & 0 & 0 & 0 & 0 & 0 & 0 & 0 & 0 & 0 & 1 & 0 & 0 & 0 & 0 & 0 & 1 & 0 & 0 & 0 & 0 & 0 & 0 & 0 & 0 & 1 & 0 & 0 & 0 & 0 & 0 & 0 & 0 & 1 & 0 & 0 & 0 & 0 & 0 & 0 & 1 & 0 & 0 & 0 & 0 & 0 & 0 \\
0 & 1 & 0 & 0 & 0 & 0 & 0 & 0 & 0 & 0 & 1 & 0 & 0 & 0 & 0 & 0 & 0 & 0 & 0 & 1 & 0 & 0 & 0 & 0 & 0 & 0 & 0 & 1 & 0 & 0 & 0 & 0 & 0 & 0 & 0 & 1 & 0 & 0 & 0 & 0 & 0 & 0 & 0 & 1 & 0 & 0 & 0 & 0 \\
0 & 0 & 1 & 0 & 0 & 0 & 0 & 0 & 0 & 0 & 0 & 0 & 0 & 0 & 1 & 0 & 1 & 0 & 0 & 0 & 0 & 0 & 0 & 0 & 1 & 0 & 0 & 0 & 0 & 0 & 0 & 0 & 0 & 0 & 0 & 0 & 1 & 0 & 0 & 0 & 1 & 0 & 0 & 0 & 0 & 0 & 0 & 0 \\
1 & 0 & 0 & 0 & 0 & 0 & 0 & 0 & 0 & 0 & 0 & 0 & 0 & 0 & 0 & 1 & 0 & 1 & 0 & 0 & 0 & 0 & 0 & 0 & 0 & 0 & 1 & 0 & 0 & 0 & 0 & 0 & 0 & 0 & 0 & 0 & 0 & 1 & 0 & 0 & 0 & 1 & 0 & 0 & 0 & 0 & 0 & 0 \\
0 & 0 & 0 & 1 & 0 & 0 & 0 & 0 & 0 & 0 & 0 & 0 & 1 & 0 & 0 & 0 & 0 & 0 & 1 & 0 & 0 & 0 & 0 & 0 & 0 & 1 & 0 & 0 & 0 & 0 & 0 & 0 & 0 & 0 & 0 & 0 & 0 & 0 & 1 & 0 & 0 & 0 & 1 & 0 & 0 & 0 & 0 & 0 \\
0 & 1 & 0 & 0 & 0 & 0 & 0 & 0 & 0 & 0 & 0 & 0 & 0 & 1 & 0 & 0 & 0 & 0 & 0 & 1 & 0 & 0 & 0 & 0 & 0 & 0 & 0 & 1 & 0 & 0 & 0 & 0 & 0 & 0 & 0 & 0 & 0 & 0 & 0 & 1 & 0 & 0 & 0 & 1 & 0 & 0 & 0 & 0 \\
0 & 0 & 0 & 0 & 0 & 0 & 0 & 0 & 0 & 1 & 0 & 0 & 0 & 0 & 0 & 0 & 1 & 0 & 0 & 0 & 0 & 1 & 0 & 0 & 0 & 0 & 0 & 0 & 0 & 0 & 0 & 0 & 1 & 0 & 0 & 0 & 0 & 0 & 0 & 0 & 1 & 0 & 0 & 0 & 1 & 0 & 0 & 0 \\
0 & 0 & 0 & 0 & 0 & 0 & 0 & 0 & 0 & 0 & 0 & 1 & 0 & 0 & 0 & 0 & 0 & 1 & 0 & 0 & 1 & 0 & 0 & 0 & 0 & 0 & 0 & 0 & 0 & 0 & 0 & 0 & 0 & 0 & 1 & 0 & 0 & 0 & 0 & 0 & 0 & 1 & 0 & 0 & 0 & 1 & 0 & 0 \\
0 & 0 & 0 & 0 & 0 & 0 & 0 & 0 & 1 & 0 & 0 & 0 & 0 & 0 & 0 & 0 & 0 & 0 & 1 & 0 & 0 & 0 & 0 & 1 & 0 & 0 & 0 & 0 & 0 & 0 & 0 & 0 & 0 & 1 & 0 & 0 & 0 & 0 & 0 & 0 & 0 & 0 & 1 & 0 & 0 & 0 & 1 & 0 \\
0 & 0 & 0 & 0 & 0 & 0 & 0 & 0 & 0 & 0 & 1 & 0 & 0 & 0 & 0 & 0 & 0 & 0 & 0 & 1 & 0 & 0 & 1 & 0 & 0 & 0 & 0 & 0 & 0 & 0 & 0 & 0 & 0 & 0 & 0 & 1 & 0 & 0 & 0 & 0 & 0 & 0 & 0 & 1 & 0 & 0 & 0 & 1 \\
0 & 0 & 0 & 0 & 1 & 0 & 0 & 0 & 0 & 0 & 0 & 0 & 0 & 0 & 1 & 0 & 0 & 0 & 0 & 0 & 0 & 1 & 0 & 0 & 0 & 0 & 0 & 0 & 1 & 0 & 0 & 0 & 0 & 0 & 0 & 0 & 1 & 0 & 0 & 0 & 0 & 0 & 0 & 0 & 1 & 0 & 0 & 0 \\
0 & 0 & 0 & 0 & 0 & 0 & 1 & 0 & 0 & 0 & 0 & 0 & 0 & 0 & 0 & 1 & 0 & 0 & 0 & 0 & 1 & 0 & 0 & 0 & 0 & 0 & 0 & 0 & 0 & 0 & 1 & 0 & 0 & 0 & 0 & 0 & 0 & 1 & 0 & 0 & 0 & 0 & 0 & 0 & 0 & 1 & 0 & 0 \\
0 & 0 & 0 & 0 & 0 & 1 & 0 & 0 & 0 & 0 & 0 & 0 & 1 & 0 & 0 & 0 & 0 & 0 & 0 & 0 & 0 & 0 & 0 & 1 & 0 & 0 & 0 & 0 & 0 & 1 & 0 & 0 & 0 & 0 & 0 & 0 & 0 & 0 & 1 & 0 & 0 & 0 & 0 & 0 & 0 & 0 & 1 & 0 \\
0 & 0 & 0 & 0 & 0 & 0 & 0 & 1 & 0 & 0 & 0 & 0 & 0 & 1 & 0 & 0 & 0 & 0 & 0 & 0 & 0 & 0 & 1 & 0 & 0 & 0 & 0 & 0 & 0 & 0 & 0 & 1 & 0 & 0 & 0 & 0 & 0 & 0 & 0 & 1 & 0 & 0 & 0 & 0 & 0 & 0 & 0 & 1
\end{array}\right)
 \] 
   }
   whose characteritic polynomial 
   $\det(xI-L({\mathbf B}_1))$ 
   is
   \[
   (x - 6) \cdot (x - 4)^{3} \cdot (x - 2)^{15} \cdot x^{29}. 
   \]
We turn 
to study the mixing time and cut off.
We see that 
the mixing time is in the order of 
$n \log n$ 
(Theorem \ref{thm:mixingtime}). 
Let 
$d_{TV}(\mu, \nu) :=
\max_{A \subset G_{n, p}} 
|
\mu (A) - \nu (A)
|$
be the total variation distance between the probability distributions 
$\mu, \nu$
on 
$G_{n, p}$. 
%
\begin{theorem}
\label{thm:cutoff}
\mbox{}\\
(1)
Let 
$c > 0$. 
Then 
we can find 
$f(c) > 0$, 
s.t. for 
$k = \lfloor n \log n + c n \rfloor$
we have 
\beq
d_{TV}
\left(
\left(
\frac {1}{np}
{\bf B}_1
\right)^k, U
\right)
&=&
f(c) + o(1), 
\quad
n \to \infty. 
\eeq
where 
\beq
f(c) \le e^{-c} + {\cal O}(e^{-2c}), 
\quad
c \to \infty.
\eeq
(2)
Suppose 
$\{ c_n \}_n$
satisfy 
\beq
\log 
\Bigl(
\log (np)
\cdot
(\log n + \alpha)
\Bigr)
\le
c_n 
<< \log n, 
\quad
\alpha > 0.
\eeq
Let 
$k := \lfloor n \log n - c_n \cdot n \rfloor$.
Then for any 
$\delta > 0$, 
we have 
\beq
d_{TV}
\left(
\left(
\frac {1}{np}
{\bf B}_1
\right)^k, U
\right)
\ge
1 - {\cal O}\left(
\frac {1}{n^{\alpha}}
\right), 
\quad
n \to \infty, 
\eeq
where 
$a_n \ll b_n$
means that 
$\lim_{n \to \infty}
a_n / b_n
= 0$
and 
$\lfloor x \rfloor
:=
\max \{ k \in {\bf Z}
\, | \,
k \le x \}$
is the  integer part of 
$x$.
\end{theorem}
%

%
{\bf Remark }\\
{\it 
(1)
In Theorem 1.2(2), 
the condition 
$\log 
\Bigl(
\log (np)
\cdot
(\log n + \alpha) 
\Bigr)
\le
c_n 
<< \log n$
for 
$\{ c_n \}$
roughly means that 
$2 \log \log n + \dfrac {
(\log p + \alpha)
}
{\log n}  < c_n \ll \log n$
for large 
$n$.
\\
(2)
The argument using the strong stationary time in 
\cite{aldousDiaconis1986}, 
\cite{Peres}
(eq.(6.16) and Proposition 7.14)
still works for this case, but Theorem 1.2 aims to study the same problem with purely combinatoric method. 
The upper bound (Theorem 1.2(1)) 
is the same as that in 
\cite{aldousDiaconis1986, Peres}. 
However, the lower bound in Theorem 1.2(2) is not good enough as in   
\cite{aldousDiaconis1986, Peres}.
}
\\

The outline 
of this paper is as follows. 
In section 2, 
we study basic properties of 
${\bf B}_k$ 
and derive a formula expressing the powers of 
${\bf B}_k$ 
in terms of orthogonal idempotents, from which we can compute the eigenvalues and corresponding eigenspaces of the left regular representation of 
${\bf B}_1$ 
explicitly. 
In section 3, 
we derive a formula computing the total variance distance beween the probability distribution of the repeated top to random shuffles and the uniform distribution.
It then  
follows that the mixing time is in the order of 
$n \log n$.
In section 4, 
we further estimate this total variation distance using the asymptotic formula for the Stirling numbers of the second kind \cite{Menon}, 
yielding a cut off statement.
In Appendix, 
we provide proofs for some elementary facts for completeness.
%

\section{Eigenvalues and their multiplicities}
We begin by 
studying some algebraic properties of 
${\bf B}_k$'s 
by which we derive the representation of the powers of 
${\bf B}_1$
(Theorem  \ref{thm:B1power}).
The following 
lemma follows from a theorem in 
\cite{tian2014generalizations}
which studies more general cases. 
However 
we present its elementary proof.
 \begin{lemma}
  \label{lem:fundamental}
  %
  %
We have 
the following formulas. \\
  (1)
  \beq
  {\bf B}_k {\bf B}_1
  &=&
  \begin{cases}
  p k {\bf B}_k + {\bf B}_{k+1} & (1 \le k \le n-1) \\
  p n {\bf B}_n  & (k = n) 
  \end{cases}
  \eeq
  (2)
  %
  \beq
 &&
  {\bf B}_k
  =
  {\bf B}_1
  ({\bf B}_1 - p {\bf I})
   ({\bf B}_1 - 2p {\bf I})
   \cdots
    ({\bf B}_1 - (k-1)p {\bf I}), 
    \quad
k=1, 2, \cdots, n
\\
%
 &&
  {\bf B}_1
  ({\bf B}_1 - p {\bf I})
   ({\bf B}_1 - 2p {\bf I})
   \cdots
    ({\bf B}_1 - (\ell-1)p {\bf I})
    = {\bf 0}, 
    \quad
    \ell > k.
  \eeq
  In particular,
 ${\mathbf B}_1,{\mathbf B}_2,\ldots,{\mathbf B}_n$ generate a
 commutative subalgebra of  ${\mathbb Q}G_{n,p}$.
 \end{lemma}
\begin{proof}
(1)
We  
suppose that 
$k \le n-1$.
The case for 
$k=n$
follows similarly. 
Then 
 we rewrite 
 ${\mathbf B}_k=\sum_{\alpha\in G_{k,p}}\alpha\shuffle W_{k,n}$
 by grouping the terms by the leading letter as follows.
 \[
  {\mathbf B}_k 
  = 
  \sum_{t\in (C_p\times[k]) \cup\{(0,k+1)\}} {\mathbf C}_{k}(t),
 \]
 where 
 ${\mathbf C}_k (t)$ 
 is the sum of the elements in 
 ${\mathbf B}_k$ 
 whose leading letter is 
 $t$. 
 For example 
 when 
 $p=2$ 
 and 
 $n=4$,
 \[
 {\mathbf C}_{2}(\bar{2}) =
 \bar{2}134 + \bar{2}314 + \bar{2}341 +
 \bar{2}\bar{1}34 + \bar{2}3\bar{1}4 + \bar{2}34\bar{1}
 \]
 and
 \[
 {\mathbf B}_2 = {\mathbf C}_2(1) + {\mathbf C}_2(\bar{1})
 +{\mathbf C}_2(2) + {\mathbf C}_2(\bar{2})
 +{\mathbf C}_2(3).
 \]
 By lemma \ref{grouping}, 
 we have
 \[
 {\mathbf C}_k (t){\mathbf B}_1=
 \begin{cases}
  {\mathbf B}_k & t\in[p]\times[k],\\
  {\mathbf B}_{k+1} & t=(0,k+1)
 \end{cases}
 %
 \]
 which yields
 \[
  {\mathbf B}_k{\mathbf B}_1 
  = 
  pk {\mathbf B}_k + {\mathbf B}_{k+1}.
 \]
(2)
From
the identity derived in (1) we have, inductively, 
\[
{\mathbf B}_{k} 
= 
{\mathbf B}_{k-1}\left({\mathbf B}_1 - p(k-1){\bf I}\right)
=
{\mathbf B}_{1}
\left({\mathbf B}_1 - p{\bf I}\right)
\left({\mathbf B}_1 - 2p{\bf I}\right)
\cdots
\left({\mathbf B}_1 - (k-1)p{\bf I}\right), 
\quad
k=1, 2, \cdots, n.
\]
Second identity in (2) 
follows similarly, by noting 
${\bf 0} = {\bf B}_n ( {\bf B}_1 - np {\bf I})$.
 \QED
\end{proof}
Let 
${k \brack a}$
(resp. ${k \brace a}$), 
where 
$k, a$
are non negative integers, 
be the Stirling number of the first kind
(resp. the second kind)
defined respectively by 
\beq
{k + 1 \brack a}
&=&
k 
{k \brack a}
+
{k \brack a-1}, 
\quad
a \ge 1, 
\quad
{0 \brack a}
=
1(a=0)
\\
 {k+1\brace a} &=& a
 {k \brace a} + {k \brace a-1}, 
 \quad
 a\ge 1, 
\quad
{0 \brace a}
=
1(a=0).
\eeq
Then 
we can express 
${\bf B}_1^k$
as a linear combination of 
${\bf B}_a$'s 
in terms of the Stirling numbers of the second kind.
Since 
these numbers defined above are the M\"obius function each other, the other way around is also possible.
%
%
 \begin{theorem}
  \label{thm:B1power}
  \begin{eqnarray}
  \label{eq:b1power}
  {\bf B}_1^k 
  &=&
  \sum_{a=0}^{n \wedge k}p^{k-a}
  { k \brace a }
  {\bf B}_a, 
  \quad
  k = 0, 1, \cdots
  \\
  {\mathbf B}_a 
  &=& 
  \sum_{i=0}^a (-p)^{a-i}{a\brack i}{\mathbf B}_1^i,
  \quad
  a = 0, 1,  \cdots, n.
 \end{eqnarray}
 with the convention that 
 ${\bf B}_1^0 = {\bf I}$ 
 and 
 $n \wedge k := \min \{ n, k \}$. 
 \end{theorem}
For proof, 
we introduce
\beq
(x)_{n, p}
&:=&
x (x-p) (x-2p) \cdots 
\bigl(
x - (n-1)p
\bigr)
\eeq
and show basic identities. 
\begin{lemma}
\label{lem:firstsecond}
\beq
&(1)&\quad
(x)_{n, p}
=
\sum_{k=0}^n 
(-p)^{n-k}
{ n \brack k }
x^k
\\
&(2)& \quad
x^n
=
\sum_{k=0}^n
p^{n-k} 
{n \brace k} 
(x)_{k, p}.
\eeq
\end{lemma}
\begin{proof}
It suffices to substitute 
$(x)_{n, p}
=
p^n 
\left(
\dfrac xp
\right)_n$
into the following well-known formulas. 
\beq
(x)_n
&=&
\sum_{k=0}^n
(-1)^{n-k}
{ n \brack k } 
x^k
\\
x^n
&=&
\sum_{k=0}^n
{n \brace k} 
(x)_k.
\eeq
\QED
\end{proof}
Theorem \ref{thm:B1power}
follows immediately from Lemma 
\ref{lem:firstsecond}.
\\
{\it Proof of Theorem \ref{thm:B1power}}\\
By 
Lemma \ref{lem:fundamental}(2), 
we have ${\bf B}_a
=
({\bf B}_1)_{a, p}$
for 
$a =1, 2, \cdots, n$,
and 
$({\bf B}_1)_{a, p}  = {\bf 0}$
for 
$a \ge n+1$.
Taking 
$x = {\bf B}_1$
in Lemma \ref{lem:firstsecond}
yields
\beq
{\bf B}_a
&=&
({\bf B}_1)_{a, p}
=
\sum_{k=0}^a
(-p)^{a-k}
{a \brack k}
{\bf B}_1^k, 
\quad
a=1, 2, \cdots, n
\\
{\bf B}_1^k
&=&
\sum_{a=0}^k
p^{k-a}
{k \brace a} 
({\bf B}_1)_{a, p}
=
\sum_{a=0}^{n \wedge k}
p^{k-a}
{k \brace a}
{\bf B}_a,
\quad
k = 1, 2, \cdots,
\eeq
Besides, 
we can explicitly see that they are valid also for 
$a = 0$
and 
$k = 0$. 
\QED
\\

{\bf  Remark }
{\it 
For given 
$p \in {\bf N}$, 
${k \brack a}_p
:=
p^{k-a} 
{ k \brack a }$ 
and
${k \brace a}_p
:=
p^{k-a} 
{ k \brace a }$
satisfy the recursion equation and the M\"obius relation similar to the usual one, so that we can regard them as a 
$p$-version of the Stirling numbers. 
\beq
&&
{k+1 \brack a}_p
=
pk
{k \brace a}_p
+
{k \brack a-1}_p, 
\quad
{0 \brack 0}_p=1
\\
&&
{k+1 \brace a}_p
=
pa
{k \brace a}_p
+
{k \brace a-1}_p, 
\quad
{0 \brace 0}_p=1
\\
&&
\sum_j
(-1)^{n-j}
{ n \brack j }_p
{ j \brace i }_p
=
\delta_{n, i}.
\eeq
However, 
${ k \brack a }_p$
is different from the Stirling-Frobenius cycle number of parameter 
$p$
which appears in the analysis of a 
$p$-version
of the riffle shuffle
\cite{NakanoSadahiroRiffle, NakanoSadahiroDet}
}
; 
For a generalized riffle shuffle (i.e., the riffle shuffle on 
$G_{n, p}$), 
the multiplicity of eigenvalues are equal to the Stirling-Frobenius cycle number
\cite{NakanoSadahiroDet}.\\

We define the elements 
 ${\bf e}_i$ 
 of the group algebra 
 ${\mathbb Q}G_{n,p}$ 
 by
 \begin{equation}
  \label{eq:idempotent}
   {\bf e}_i =\frac{1}{i!} \sum_{a=i}^{n}\frac{(-1)^{a-i}}{p^a(a-i)!}{\bf B}_a,
 \end{equation}
for 
$i=0,1,\ldots,n$.
Then 
the powers of 
${\bf B}_1$
are expressed in terms of 
$\{ {\bf e}_i \}$.
%

\begin{theorem}
\begin{equation}
\label{eq:Bpower}
{\bf B}_1^k = \sum_{i=0}^n (pi)^k {\bf e}_i, 
\quad
k=0, 1,\ldots,
\end{equation}
\end{theorem}
%
%
%
\begin{proof}
We use 
the following idendity 
\cite{GarsiaLecNote}
\[
  {k\brace a} =
   [t^k]
   \left(
   \frac{k!}{a!}\left(e^{t}-1\right)^a
   \right). 
\]
In fact,
by Taylor's expansion, 
\beq
\frac{k!}{a!}
\left(
e^{t}-1
\right)^a
&=&
\frac {k!}{a!}
\left(
\frac {t}{1!}
+
\frac {t^2}{2!}
+
\cdots
\right)
\left(
\frac {t}{1!}
+
\frac {t^2}{2!}
+
\cdots
\right)
\cdots
\left(
\frac {t}{1!}
+
\frac {t^2}{2!}
+
\cdots
\right).
\eeq
Taking 
the coefficient of 
$t^k$
leads us to this formula : 
\beq
\lbrack t^k \rbrack
\frac{k!}{a!}
\left(
e^{t}-1
\right)^a
&=&
\sum_{
\substack{
k_1 + \cdots + k_a = k \\ 
k_1, \cdots, k_a \ge 1
}
}
\frac {k!}{
k_1! k_2! \cdots k_a!
}
\cdot
\frac {1}{a!}
=
{k \brace a}.
\eeq
By the binomial theorem, 
\beq
{k \brace a}
&=&
[t^k]
\frac {k!}{a!}
\sum_{i=0}^a
\left(
\begin{array}{c}
a \\ i
\end{array}
\right)
(-1)^{a-i}
e^{it}
=
\sum_{i=0}^a
\frac {(-1)^{a-i}}{i! (a-i)!}
\cdot
i^k.
\eeq
We note that 
this formula is valid also for 
$k=0$. 
Using 
this equation in
(\ref{eq:b1power})
and changing the order of summation yield the conclusion.
We 
note that in 
(\ref{eq:b1power}), the summation 
$\sum_{a=0}^{n \wedge k}$
may be replaced by 
$\sum_{a=0}^n$, 
since
${k \brace a} = 0$
for 
$a > k$.
\beq
{\bf B}_1^k
&=&
\sum_{a=0}^n
p^{k-a}
\sum_{i=0}^a
\frac {(-1)^{a-i}}{i! (a-i)!}
\cdot
i^k
{\bf B}_a
=
\sum_{i=0}^n
(ip)^k
\sum_{a=i}^n
\frac {1}{p^a}
\cdot
\frac {(-1)^{a-i}}
{ i! (a-i) ! }
{\bf B}_a
=
\sum_{i=0}^n
(ip)^k
{\bf e}_i.
\eeq
\QED
\end{proof}
{\bf Remark  }
~The argument 
of proof of Lemma \ref{eigenspace}
and 
eq.(\ref{eq:Bpower})
imply that 
$\{ {\bf e}_i \}$
is the orthogonal idempotents : 
\begin{equation}
\sum_{i=0}^n {\bf e}_i = I, 
\quad
{\bf e}_i  {\bf e}_j = \delta_{i, j} {\bf e}_i.
\label{idempotents}
\end{equation}
\\

The following lemma 
is stated in 
\cite{assaf2010cyclic} 
which can be proved by standard inclusion-exclusion principle
\cite{StanleyEC1}.
\begin{lemma}
{\rm \cite{assaf2010cyclic}}
  Let 
  $D_{n,p}$ 
  be the number of derangements in 
  $G_{n,p}$.
  Then,
  \[
   D_{n,p} 
   = 
   p^n n!\sum_{k=0}^n
   \frac{(-1)^k}{p^k k!}.
  \]
\end{lemma}
{\it Proof of Theorem $\ref{thm:main}$ }
Let 
 $E_i$ 
 be the matrices defined by
 \[
  E_i = L(e_i)
 \]
  for 
  $i=0,1,\ldots,n$.
 By transforming 
 both sides of eq.
 $(\ref{eq:Bpower})$ 
 by 
 $L$,
 powers of 
 $L( {\bf B}_1 )$
 can be represented in terms of 
 $E_i$'s. 
\[
 L({\bf B}_1)^k
 =
 \sum_{i=0}^n
 (pi)^k E_i. \]
Then 
by Lemma \ref{eigenspace} in Appendix it follows that  
$\{ (ip) \}_{i=0}^n$
are the eigenvalues and the range 
$Ran \, E_i$
of 
$E_i$ 
(if it is nonzero)
are the corresponding eigenspaces.
 Since each 
 ${\bf B}_a$ 
 contains exactly one identity permutation,
 we have 
 ${\rm Trace }\, L({\mathbf B}_a) = |G_{n,p}| = p^nn!$.
 Then 
 we compute
 \begin{eqnarray*}
  {\rm Trace} \;
  E_i 
  & = & 
  \frac{1}{i!}
  \sum_{a=i}^n
  \frac{(-1)^{a-i}}{p^a(a-i)!}
  p^nn!
 =
  {n\choose i}
  p^{n-i}(n-i)!
  \sum_{b=0}^{n-i}
  \frac{(-1)^{b}}{p^{b}b!}
  =
  {n\choose i}
  D_{n-i,p}
 \end{eqnarray*}
 which is the number
 of elements of 
 $G_{n,p}$ 
 with 
 $i$ 
 fixed points.
\QED\\

We 
have analogous formulas for 
${\bf B}_a$. 
%

\begin{corollary}
 \begin{equation}
  \label{eq:BapowerE}
 {\mathbf B}_a^k = \sum_{i=a}^n\left(p^aa!{i\choose a}\right)^k {\bf e}_i, 
 \quad
 k = 0, 1, \ldots
 \end{equation}
 Therefore, 
 the eigenvalues of  
 $L({\mathbf B}_a)$ 
 are
 \[
  0,p^aa!,~p^aa!{a+1\choose a},p^aa!{a+2\choose a},\ldots, p^aa!{n\choose a}.
 \]
The multiplicity 
of the eigenvalue 
$p^aa!{i\choose a}$ 
is same as that of the eigenvalues of 
${\bf B}_1$. 
\end{corollary}
\begin{proof}
It suffices to show 
  \[
  {\mathbf B}_a 
  = 
  \sum_{i=a}^np^aa!{i\choose a}{\bf e}_i, 
  \quad
  a = 0, 1, \cdots, n.
  \]
and then use eq.(\ref{idempotents}).
In order for that, 
we aim to express 
${\bf B}_a$
in terms of 
${\bf e}_i$'s
by using eq.
(\ref{eq:idempotent})
\begin{equation}
{\bf e}_i
=
\sum_{b=i}^n
\frac {1}{ p^i i! }
\frac {(-1)^{b-i}}
{
p^{b-i} (b-i)! 
}
{\bf B}_b
=
\sum_{b=i}^n
\frac {1}{ p^i i! }
[ x^{b-i} ]
\left(
e^{- \frac xp}
\right)
{\bf B}_b.
\label{eq:one}
\end{equation}
The 
``reciprocal" 
of these coefficients are equal to 
\beq
i! p^i
[ x^{i-a} ]
\left(
e^{\frac xp} 
\right)
&=&
i ! p^i
\frac {1}{ (i-a)! }
[ x^{i-a} ]
\left(
\frac xp 
\right)^{i-a}
=
p^a a!
{ i \choose a }
\eeq
which satisfy 
\beq
&&
\sum_{i = a}^b
\left(
e^{\frac xp}
\right) 
[ x^{ i- a } ] 
\left(
e^{ - \frac xp}
\right) 
[ x^{ b - i } ] 
=
\left(
e^{ \frac xp }
\cdot
e^{- \frac xp}
\right)
[ x^{b - a} ]
=
\delta_{a, b}.
\eeq
Thus, 
applying  
$\sum_{i=a}^n 
i! p^i 
[ x^{i-a} ] 
\left(
e^{\frac xp} 
\right)
$
on both sides of 
(\ref{eq:one}) 
yields 
\beq
\sum_{i=a}^n 
i! p^i 
\left(
e^{\frac xp} 
\right)
[ x^{i-a} ] 
{\bf e}_i
&=&
\sum_{i=a}^n 
i! p^i 
\left(
e^{\frac xp} 
\right)
[ x^{i-a} ] 
\sum_{b=i}^n
\frac {1}{ p^i i! }
\left(
e^{- \frac xp}
\right)
[ x^{b-i} ]
{\bf B}_b
\\
\sum_{i=a}^n
p^a a! 
{ i \choose a }
{\bf e}_i
&=&
\sum_{i=a}^n 
\sum_{b=i}^n
\left(
e^{\frac xp} 
\right)
[ x^{i-a} ] 
\left(
e^{- \frac xp}
\right)
[ x^{b-i} ]
1 
\left(
0 \le a \le i \le b \le n
\right)
{\bf B}_b
\\
&=&
\sum_{ b=a }^n
\sum_{i=a}^b
\left(
e^{\frac xp} 
\right)
[ x^{i-a} ] 
\left(
e^{- \frac xp}
\right)
[ x^{b-i} ]
{\bf B}_b
\\
&=&
\sum_{b=a}^n
\delta_{a, b} {\bf B}_b
=
{\bf B}_a.
\eeq
\QED
\end{proof}
Here, 
we introduce the {\em generalized Stirling} numbers, which arise in the Boson normal ordering problem
\cite{BPS}.
%
\beq
S_{r, s}(n, k)
&:=&
\frac {(-1)^k}{k!}
\sum_{p=s}^k
(-1)^p
{ k \choose p }
\prod_{j=1}^n
\Bigl(
p + (j-1)(r-s)
\Bigr)^{ \underline{s} }
\\
where
\quad
m^{ \underline{s} }
&:=&
m (m-1) \cdots (m-s+1).
\eeq
Then 
we obtain 
${\bf B}_a$-analogue of Theorem \ref{thm:B1power}.
%
%
\begin{theorem}
 \[
 {\bf  B}_a^k 
 = 
 \sum_{b=a}^n p^{ka-b}S_{a,a}(k,b){\mathbf B}_b.
 \]
\end{theorem}
\begin{proof}
Using 
$p^{ \underline{a} }
=
a!
{ p \choose a }$
in the definition, we have 
\beq
S_{a,a}(k,b)
&=&
\frac {(-1)^b}{b!}
\sum_{i=a}^b
(-1)^i
{ b \choose i }
\Bigl(
a! 
{ i \choose a }
\Bigr)^k.
\eeq
Using 
\beq
{\bf e}_i
&=&
\frac {1}{i!}
\sum_{b=i}^n
\frac {(-1)^{b-i}}
{
p^b (b-i)!
}
{\bf B}_b
=
\sum_{b=i}^n
\frac {1}{p^b}
\frac {(-1)^{b-i}}{b!}
{ b \choose i }
{\bf B}_b
\eeq
in eq. 
(\ref{eq:BapowerE})
yields 
\beq
{\bf B}_a^k
&=&
\sum_{i=a}^n
\Bigl(
p^a a! 
{ i \choose a }
\Bigr)^k
\sum_{b=i}^n
\frac {1}{p^b}
\frac {(-1)^{b-i}}{b!}
{ b \choose i }
{\bf B}_b
\\
&=&
\sum_{b=a}^n
\frac {p^{ka-b}}{b!}
\sum_{i=a}^b
(-1)^{b-i}
{ b \choose i }
\Bigl(
a! 
{ i \choose a }
\Bigr)^k
{\bf B}_b
\\
&=&
\sum_{b=a}^n
p^{ka-b}
S_{a,a} (k, b)
{\bf B}_b. 
\eeq
\QED
\end{proof}
{\bf Remark }
{\it 
An asymptotic formula for 
$S_{a,a}(k,b)$
would yield a cut off statement for the shuffles corresponding to 
${\bf B}_a$.
}
%

\section{Mixing time}
In this section, 
we consider the mixing time of the top to random shuffle on the colored permutation.
Although 
the state space of the Markov chain is $p^n$ 
times larger than the ordinary top to random shuffle on 
${\mathfrak S}_n$,
it turns out that the mixing time does not differ significantly ; in fact, our bound on the mixing time is independent of 
$p$, 
for 
$p > 1$. 
Let 
${\mathbf B}
=
\sum_{w\in G_{n,p}} c_w w \in {\mathbb Q}G_{n,p}$ 
be an element of the group algebra 
${\mathbb Q}G_{n,p}$.
We define the 
$L^1$-norm 
$|{\mathbf B}|$ 
of 
${\mathbf B}$ 
by
\[
|{\mathbf B}| 
= 
\sum_{w\in G_{n,p}}|c_w|.
\]
Then 
it can easily be confirmed
\[
 |{\mathbf B}_a| = {n\choose a}p^aa!.
\]
and since these elements in 
$G_{n, p}$
consisting of 
${\bf B}_{a-1}$
are contained by those in 
${\bf B}_a$, 
\[
|{\mathbf B}_a-{\mathbf B}_{a-1}| 
=
|{\mathbf B}_a| - |{\mathbf B}_{a-1}|.
\]
A probability distribution 
${\bf P}$ 
over 
$G_{n,p}$ 
can be regarded as an element of 
${\mathbb R}G_{n,p}$, 
i.e., 
${\bf P}$ 
can be expressed as
\[
 {\bf P} 
 = 
 \sum_{w\in G_{n,p}} p_w w,
\]
where 
$p_w\geq 0$ 
and
$\sum_{w\in G_{n,p}}p_w = 1$.
Let 
${\bf P},{\bf Q} \in {\mathbb R}G_{n,p}$ 
be two probability distributions on 
$G_{n,p}$. 
%
Then 
$d_{\rm TV}({\mathbf P}, {\mathbf Q})$ 
is equal to
\[
 d_{\rm TV}({\bf P},{\bf Q}) 
 = 
 \frac{1}{2}
 \left|{\mathbf P}-{\mathbf Q}\right|.
\]
\begin{theorem}
\label{thm:mixingtime}
Let 
$p$ 
be greater than 
$1$ 
and let 
$U$ 
be the uniform distribution over 
$G_{n,p}$, 
that is,
\[
 U = \sum_{w\in G_{n,p}}\frac{1}{p^nn!}w.
\]
(1)
The 
total variation distance between the distribution of 
$k$-repeated top to random shuffle and the uniform distribution is bounded above by 
\begin{eqnarray*}
d_{TV}
\left(
\left(\frac{1}{pn}{\mathbf B}_1\right)^k, 
\,
U
\right)
& \leq &
1-\frac{{k\brace n}n!}{n^k}.
\end{eqnarray*}
(2)
Let 
$\epsilon > 0$.
Then 
$d_{TV}
\left(
\left(
\dfrac{1}{pn}{\mathbf B}_1
\right)^k, 
\,
U
\right)
<
\epsilon$
for 
$k\geq n\log n 
+ 
n\log\dfrac{-1}{\log(1-\varepsilon)}$.
\end{theorem}
To prove Theorem 
\ref{thm:mixingtime}, 
we need 
the 
following lemma which gives us the TV distance between the distribution of 
$k$-repeated top to random shuffles and the uniform distribution. 
\begin{lemma}
\label{lem:TV}
Let 
\beq
A 
&:=&
\min
\left\{
a \, \middle| \,
\frac {1}{p^n n!}
>
\frac {1}{(np)^k}
\sum_{b=a}^n
p^{k-b}
\left\{
\begin{array}{c}
k \\ b
\end{array}
\right\}
\right\}.
\eeq
Then 
\beq
d_{TV}
\left(
\left(
\frac {1}{ pn }
B_1
\right)^k, 
U
\right)
& = &
\sum_{a \ge A}
\left(
\frac {1}{ p^n n! }
-
\frac {1}{ p^k n^k }
\sum_{b = a }^n
p^{k-b}
{ k \brace b }
\right)
(
|{\bf B}_a| - | {\bf B}_{a-1} |
).
\eeq
\end{lemma}
\begin{proof}
Let 
\beq
{\bf C}_a 
&:=& {\bf B}_a - {\bf B}_{a-1}, 
\;
{\bf C}_0 := {\bf B}_0.
\\
x_a
&:=&\frac {1}{
( pn )^k }
\sum_{b=a}^n
p^{k-b}
{ k \brace b }, 
\quad
a = 0, 1, \cdots, n. 
\eeq
Then 
\beq
\left(
\frac {1}{ pn }
{\bf B}_1
\right)^k
&=&
\frac {1}{
( pn )^k }
\sum_{b=1}^n
p^{k-b}
{ k \brace b }
\sum_{a=1}^b
{\bf C}_a
+
\frac {1}{
( pn )^k }
\sum_{b=1}^n
p^{k-b}
{ k \brace b }
{\bf C}_0
\\
&=&
\sum_{a=1}^n
\frac {1}{
( pn )^k }
\sum_{b=a}^n
p^{k-b}
{ k \brace b }
{\bf C}_a
+
\frac {1}{
( pn )^k }
\sum_{b=0}^n
p^{k-b}
{ k \brace b }
{\bf C}_0
\\
&=&
\sum_{a=0}^n
x_a 
{\bf C}_a, 
\quad
k = 0, 1, \cdots, 
\eeq
Where 
we used 
${ k \brace 0 } = 0$ 
for 
$k \ne 0$.
Similarly, using 
${\bf B}_n
=
\sum_{a=0}^n {\bf C}_a$
we have 
\beq
U
&=&
\sum_{a=0}^n
y_a {\bf C}_a,
\quad
y_a := 
\frac {1}{
n! p^n
}.
\eeq
Hence 
\beq
d_{TV}
\left(
\left(
\frac {1}{ pn }
{\bf B}_1
\right)^k, 
U
\right)
&=&
\frac 12
\sum_{a=0}^n
| x_a - y_a | | {\bf C}_a |
=
\sum_{a \, : \, y_a > x_a}
(x_a - y_a ) | {\bf C}_a |
\eeq
where the second equality follows from the fact that 
$\sum_{a=0}^n y_a | {\bf C}_a |
=
\sum_{a=0}^n x_a | {\bf C}_a |
= 1$
(Proposition 4.2 \cite{Peres}).
Since 
$x_a$
is monotonically decreasing and 
$y_a$
is constant, and since 
$A :=
\min 
\{ a \, | \, 
y_a > x_a 
\}$, 
we have 
\beq
d_{TV}
\left(
\left(
\frac {1}{ pn }
{\bf B}_1
\right)^k, 
U
\right)
&=&
\sum_{a \ge A}
(y_a - x_a) | {\bf C}_a |
\\
& = &
\sum_{a \ge A}
\left(
\frac {1}{ p^n n! }
-
\frac {1}{ p^k n^k }
\sum_{b = a }^n
p^{k-b}
{ k \brace b }
\right)
(
|{\bf B}_a| - | {\bf B}_{a-1} |
). 
\eeq
\QED
\end{proof}
%
%
%
%
{\it Proof of Theorem \ref{thm:mixingtime} }
(1)
By Theorem  
\ref{thm:B1power}, 
we have 
  \[
  \left(
  \frac{1}{pn}{\mathbf B}_1
  \right)^k
  =
  \frac{1}{p^kn^k}
  \sum_{a=1}^n p^{k-a}{k\brace a}{\mathbf B}_a.
  \]
  Therefore, 
  if we let 
  $A$ 
  be the smallest of the integers 
  $a$ 
  such that
  \[
   \frac{1}{p^nn!} 
   >
   \frac {1}
    {p^kn^k}
  \sum_{b=a}^np^{k-b}{k\brace b},
  \]
  %
we have, by Lemma \ref{lem:TV}, 
\begin{eqnarray*}
d_{\rm TV}
\left(
\left(\frac{1}{pn}{\mathbf B}_1
\right)^k,
U
\right)
& = &
\sum_{a \ge A}
\left(
\frac {1}{ p^n n! }
-
\frac {1}{ p^k n^k }
\sum_{b = a }^n
p^{k-b}
{ k \brace b }
\right)
(
|B_a| - | B_{a-1} |
)
\\
& \leq &
\sum_{a\geq A}
\left(
\frac{1}{p^nn!}
-
\frac {1}{p^kn^k}
\cdot
p^{k-n}{k\brace n}
\right)
\left(|{\mathbf B}_a|-|{\mathbf B}_{a-1}|
\right)\\\\
& \leq &
\left(
\frac{1}{p^nn!}
-
\frac {1}{p^kn^k}
\cdot
p^{k-n}{k\brace n}
\right)
|{\mathbf B}_n|
=
1-\frac{{k\brace n}n!}{n^k}.
\end{eqnarray*}
(2)
This immediately follows from 
(1) above and 
Lemma \ref{lem:limit} in Appendix.
\QED
%

\section{Cut off}
In this section
we prove 
Theorem \ref{thm:cutoff}.
First of all, 
it is easy to show the upper bound in Theorem 1.2(1). 
In fact, 
by 
Theorem \ref{thm:mixingtime} and Lemma \ref{lem:limit} 
we have, for 
$k = \left\lfloor n \log n + cn \right\rfloor$, 
\beq
d_{TV}
\left(
\frac {1}{(np)^k}
B_1^k, U
\right)
&\le&
1-
\exp[ -  n e^{- \frac kn}]
(1+o(1))
\stackrel{n \to \infty}{\to}
1 - \exp [ - e^{-c} ].
\eeq
It then 
suffices to compute the Taylor's expansion of 
$\exp [ -e^{-c} ]$
to prove Theorem 1.2(1). 
For the lower bound, let
\beq
X
:=
\lfloor
n - \log n 
\rfloor.
\eeq
We shall 
divide into two cases : 
Case 1 : $A \le X$ 
and 
Case 2 : $X \le A$. 
%
\subsection{
Case 1 : ($A \le X$)
}
We first substitute 
$\sum_{a \ge A}$
for 
$\sum_{a \ge X}$
in the formula in Lemma 
\ref{lem:TV}. 
Using 
$|B_a|
=
p^a a! 
{ n \choose a }$, 
we have 
\begin{eqnarray}
&&
d_{TV}
\left(
\frac {1}{(np)^k}
B_1^k, U
\right)
\nonumber
\\
& \ge &
\sum_{a \ge X}
\left(
\frac {1}{p^n n!}
-
\frac {1}{(np)^k}
\sum_{b=a}^n
p^{k-b}
\left\{
\begin{array}{c}
k \\ b
\end{array}
\right\}
\right)
\left(
|B_a| - |B_{a-1}|
\right)
\nonumber
\\
& \ge &
\sum_{a \ge X}
\left(
\frac {1}{p^n n!}
-
\frac {1}{(np)^k}
\sum_{b=X}^n
p^{k-b}
\left\{
\begin{array}{c}
k \\ b
\end{array}
\right\}
\right)
\left(
|B_a| - |B_{a-1}|
\right)
\nonumber
\\
&=&
\left(
1
-
\frac {n!}{n^k}
\sum_{b=X}^n
p^{n-b}
\left\{
\begin{array}{c}
k \\ b
\end{array}
\right\}
\right)
\left(
1
-
\frac {1}{
(n - X + 1)!
\cdot
p^{n-X+1}
}
\right)
\label{sharp}
\\
&=:&
(1-C)(1-D).
\nonumber
\end{eqnarray}
We aim to show 
$C, D = {\cal O}(n^{-\alpha})$
below. 
%

\subsubsection{
Estimate for  
$C$}
We use 
the following property of the Stirling numbers of the second kind \cite{Dobson}: for given 
$k \in {\bf N}$, 
we can uniquely find 
$r_k$
such that 
\beq
\left\{
\begin{array}{c}
k \\ 1
\end{array}
\right\}
<
\left\{
\begin{array}{c}
k \\ 2 
\end{array}
\right\}
< \cdots < 
\left\{
\begin{array}{c}
k \\ r_k 
\end{array}
\right\}
\ge 
\left\{
\begin{array}{c}
k \\ r_k+1 
\end{array}
\right\}
>\cdots > 
\left\{
\begin{array}{c}
k \\ k 
\end{array}
\right\}.
\eeq
Moreover 
by eq.(1.6) in 
\cite{Menon}, 
$r_k$
satisfies 
\beq
r_k
&=&
\frac {k}{ \log k}
+
{\cal O}
\left(
k (\log k)^{- \frac 32}
\right).
\eeq
Taking 
$k
=
n \log n - n \cdot c_n$
yields
\beq
\frac {k}{\log k}
=
\frac {
n \log n 
\left(
1 - \frac {c_n}{n}
\right)
}
{
\log n + \log \log n 
+
\log 
\left(
1 - \frac {c_n}{n}
\right)
}.
\eeq
Thus 
$r_k$
quite likely is contained in the sum 
$\sum_{b = X}^n$
in eq.(\ref{sharp}).
Therefore 
we further divide into three cases, 
according to the large and small relationship, that is, 
Case (i)
$X \le n \le r_k$, 
Case (ii)
$X \le r_k \le n$, 
and 
Case (iii)
$r_k \le X \le n$.
%
%
\paragraph{
Case (i)
$X \le n \le r_k$ : 
}
Since 
$X \le b \le n$, 
we have 
${ k \brace b } \le { k \brace n }$.
Lemma \ref{lem:bound} 
and the equation  
$e^{- k/n}
=
e^{- \log n + c}
=
e^c / n$
yield 
\beq
\frac {n!}{n^k}
\sum_{b=X}^n
p^{n-b}
\left\{
\begin{array}{c}
k \\ b 
\end{array}
\right\}
& \le &
\frac {n!}{n^k}
\sum_{b=X}^n
p^{n-b}
\left\{
\begin{array}{c}
k \\ n
\end{array}
\right\}
\\
& \le &
\exp 
\left[
- e^{c_n}
\right]
\left(
1 + {\cal O}
\left(
\frac {\log n}{n^{1 - \epsilon}}
\right)
\right)
\frac {p}{p-1}
\left(
p^{n-X} - \frac 1p
\right), 
\quad
\epsilon > 0.
\eeq
In what follows, 
$\epsilon > 0$
is kept fixed.
By the assumption 
$
\log 
\Bigl(
\log (np)
\cdot
(\log n + \alpha)
\Bigr)
\le
c_n 
$
on 
$\{ c_n \}$, 
we have 
\beq
\exp 
\left[
e^{c_n}
\right]
\ge
\exp
\left[
\log (np)
\cdot
(\log n + \alpha)
\right]
=
(np)^{\log n + \alpha} 
\eeq
which leads to 
\beq
\frac {n!}{n^k}
\sum_{b=X}^n
p^{n-b}
\left\{
\begin{array}{c}
k \\ b 
\end{array}
\right\}
& \le &
\frac {1}{
(np)^{\log n + \alpha}}
\cdot
\frac {p}{p-1}
\left(
p^{\log n} - \frac 1p
\right)
=
{\cal O}
\left(
\frac {1}{
n^{ \log n + \alpha  }
}
\right).
\eeq
%
%
\paragraph{
Case (ii)
$X \le r_k \le n$ : }
We define
$g(n)$ 
by the equation 
$r_k =: n - g(n)$.
$X \le r_k \le n$
implies 
$0 \le g(n) \le \log n$.
We then compute 
\beq
\frac {n!}{n^k}
\sum_{b=X}^n
p^{n-b}
\left\{
\begin{array}{c}
k \\ b 
\end{array}
\right\}
& \le &
\frac {n!}{n^k}
\sum_{b=X}^n
p^{n-b}
\left\{
\begin{array}{c}
k \\ r_k
\end{array}
\right\}
\\
&=&
\left(
\frac {r_k}{n}
\right)^k
\frac {n!}{r_k!}
\cdot
\frac {r_k!}{r_k^k}
\left\{
\begin{array}{c}
k \\ r_k 
\end{array}
\right\}
\frac {p}{p-1}
\left(
p^{n-X} - \frac 1p
\right)
\\
&=:&
E \cdot F \cdot G \cdot
\frac {p}{p-1}
\left(
p^{n-X} - \frac 1p
\right).
\eeq
We shall estimate each factors 
$E, F, G$
below. 
$E$ and $F$ 
are easy : 
\begin{eqnarray}
E
&:=&
\left(
\frac {r_k}{n}
\right)^k
=
\left(
1 - \frac {g(n)}{n}
\right)^k
=
\left(
1 - \frac {g(n)}{n}
\right)^{n \log n - c_n \cdot n}
\le 1
\label{E}
\\
F 
&:=&
\frac {n!}
{(n - g(n))!}
\le
n^{g(n)}
\le
n^{ \log n}.
\label{F}
\end{eqnarray}
To estimate 
$G$, 
we first use 
Lemma \ref{lem:bound} in Appendix.
\beq
G
&:=&
\frac {r_k!}{r_k^k}
\left\{
\begin{array}{c}
k \\ r_k 
\end{array}
\right\}
\le
\exp
\left[
- r_k 
e^{- \frac {k}{r_k}}
\right]
e^{ \frac {1}{2k} }
\left(
1 + {\cal O}
\left(
\frac {\log n}{n^{1 - \epsilon}}
\right)
\right)
\\
&=&
\exp
\left[
- n
\left(
1 - \frac {g(n)}{n}
\right)
\left(\frac 1n \right)^{\frac {1}{1 - \frac {g(n)}{n}}}
\cdot
\left(
e^{c_n}
\right)^{
\frac {1}{1 - \frac {g(n)}{n}}
}
\right]
\left(
1 + {\cal O}
\left(
\frac {\log n}{n^{1 - \epsilon}}
\right)
\right)
.
\eeq
Each factors 
in the exponential satisfy 
\beq
\left(
\frac 1n
\right)^{
\frac {1}{
1 - \frac {g(n)}{n}
}
}
&=&
\frac 1n
\left(
1 +
{\cal O}
\left(
\frac {(\log n)^2}{n}
\right)
\right)
\\
\left(
e^{c_n}
\right)^{
\frac {1}{1 - \frac {g(n)}{n}}
}
&=&
\left(
1 + 
{\cal O}\left(
\dfrac {(\log n)^2}{n}
\right)
\right)
e^{c_n}
\eeq
so that
\begin{equation}
G
\le 
\exp
\left[
- e^{c_n}
\left(
1 + {\cal O}
\left(
\frac {(\log n)^2}{n}
\right)
\right)
\right]
\label{G}.
\end{equation}
By 
(\ref{E}, \ref{F}, \ref{G}), 
we have 
\beq
\frac {n!}{n^k}
\sum_{b=X}^n
p^{n-b}
\left\{
\begin{array}{c}
k \\ b 
\end{array}
\right\}
& \le &
1 \cdot n^{\log n}
\cdot
\exp
\left[
- e^{c_n}
\left(
1 + 
{\cal O}
\left(
\frac {(\log n)^2}{n}
\right)
\right)
\right]
\frac {p}{p-1}
\left(
p^{\log n} - \frac 1p
\right)
\\
&=&
(np)^{\log n}
\exp 
\left[
- e^{c_n}
\left(
1 + 
{\cal O}
\left(
\frac {(\log n)^2}{n}
\right)
\right)
\right]
\frac {p}{p-1}
\left(
1 - \frac {1}{p^{\log n+1}}
\right).
\eeq
Therefore 
the condition 
$\log 
\Bigl(
(\log np)(\log n + \alpha) 
\Bigr)
\le c_n$
on 
$\{ c_n \}$
yields 
\beq
&&
\frac {n!}{n^k}
\sum_{b=X}^n
p^{n-b}
\left\{
\begin{array}{c}
k \\ b 
\end{array}
\right\}
\le
\frac {
(np)^{\log n}
}
{
(np)^{\log n + \alpha}
}
\cdot
\exp
\left[
(\log n)^2
{\cal O}
\left(
\frac {(\log n)^2}{n}
\right)
\right]
=
{\cal O}
\left(
\frac {1}{
n^{ \alpha }
} 
\right).
\eeq
%

%
\paragraph{
Case (iii)
$r_k \le X  \le n$ : }
We proceed as Case (ii) : 
\beq
\frac {n!}{n^k}
\sum_{b=X}^n
p^{n-b}
\left\{
\begin{array}{c}
k \\ b 
\end{array}
\right\}
& \le &
\frac {n! p^n}{n^k}
\left\{
\begin{array}{c}
k \\ X
\end{array}
\right\}
\sum_{b=X}^n
p^{-b}
\\
&=&
\frac {n!}{n^k}
\cdot
\frac {X^k}{X!}
\cdot
\frac {X!}{X^k}
\left\{
\begin{array}{c}
k \\ X
\end{array}
\right\}
p^n 
\sum_{b=X}^n
p^{-b}
\\
& \le &
\left(
\frac Xn
\right)^k
\frac {n!}{X!}
\exp
\left[
- X e^{ - \frac kX} 
\right]
\cdot
\frac {p^2}{p-1}
\left(
p^{\log n} - \frac 1p
\right)
\\
&=:&
E \cdot F \cdot G \cdot
\frac {p^2}{p-1}
\left(
p^{\log n} - \frac 1p
\right)
\eeq
and aim to estimate each factors 
$E, F, G$.
However, 
this is reduced to replacing  
$g(n)$ 
in Case (ii) by  
$\log n$, 
and hence a similar argument yields  
\beq
&&
\frac {n!}{n^k}
\sum_{b=X}^n
p^{n-b}
\left\{
\begin{array}{c}
k \\ b 
\end{array}
\right\}
=
{\cal O}
\left(
\frac {1}{n^{ \alpha }} 
\right).
\eeq
%

\subsubsection{
Estimate for 
$D$}
Using 
$(\log n+1)!
\approx
n^{\log \log n}$
and 
$p^{\log n} = n^{\log p}$
directly leads us to the conclusion : 
\beq
D
=
\frac {1}{
(\log n+1) ! p^{\log n+1}
}
=
{\cal O}
\left(
n^{- \log \log n - \log p}
\right).
\eeq
%

%
\subsection{
Case 2 : ($A \ge X$)
}
We first compute 
\beq
d_{TV}
\left(
\frac {1}{(np)^k}
B_1^k, U
\right)
&=&
\sum_{a \ge A}
\left(
\frac {1}{p^n n!}
-
\frac {1}{(np)^k}
\sum_{b=a}^n
p^{k-b}
\left\{
\begin{array}{c}
k \\ b
\end{array}
\right\}
\right)
\left(
|B_a| - |B_{a-1}|
\right)
\\
& \ge &
\sum_{a \ge A}
\left(
\frac {1}{p^n n!}
-
\frac {1}{(np)^k}
\sum_{b=A}^n
p^{k-b}
\left\{
\begin{array}{c}
k \\ b
\end{array}
\right\}
\right)
\left(
|B_a| - |B_{a-1}|
\right)
\\
&=&
\left(
\frac {1}{p^n n!}
-
\frac {1}{(np)^k}
\sum_{b=A}^n
p^{k-b}
\left\{
\begin{array}{c}
k \\ b
\end{array}
\right\}
\right)
\sum_{a \ge A}
\left(
|B_a| - |B_{a-1}|
\right)
\\
& \ge &
\left(
\frac {1}{p^n n!}
-
\frac {1}{(np)^k}
\sum_{b=X}^n
p^{k-b}
\left\{
\begin{array}{c}
k \\ b
\end{array}
\right\}
\right)
\sum_{a \ge A}
\left(
|B_a| - |B_{a-1}|
\right)
\\
&=&
\left(
1 - 
\frac {n!}{n^k}
\sum_{b=X}^n
p^{n-b}
\left\{
\begin{array}{c}
k \\ b 
\end{array}
\right\}
\right)
\left(
1 - 
\frac {1}{
(n - A +1)! \, p^{n-A+1}
}
\right)
\\
&=:&
(1 - C) (1 - D')
\eeq
where 
$D'$
is equal to 
$D$
but 
$X$
is substituted by 
$A$. 
$C$
satifies the same estimate as in Case 1. 
In fact, 
we did not use the fact that 
$A \le X$
to estimate 
$C$ 
in Case 1.
Hence
\beq
C
=
\frac {n!}{n^k}
\sum_{b=X}^n
p^{n-b}
\left\{
\begin{array}{c}
k \\ b 
\end{array}
\right\}
=
{\cal O}
\left(
\frac {1}{
n^{ \alpha }
}
\right).
\eeq
To estimate 
$D'$, 
let  
$M := n - A$.
It then 
suffices to show 
$M \ge (Const.) \log n$.
%
%
By the definition of 
$A$, 
\beq
\frac {1}{p^n n!}
\le
\frac {1}{(np)^k}
\sum_{b = A-1}^n
p^{k-b}
\left\{
\begin{array}{c}
k \\ b 
\end{array}
\right\}
\eeq
so that we have 
\begin{equation}
1 
\le 
\frac {n!}{n^k}
\sum_{b = A-1}^n
p^{n-b}
\left\{
\begin{array}{c}
k \\ b 
\end{array}
\right\}
\le
\sum_{b = A-1}^n
p^{n-b}
\begin{cases}
\dfrac {n!}{n^k}
\left\{
\begin{array}{c}
k \\ n 
\end{array}
\right\}
& (X \le n \le r_k) 
\\
\dfrac {n!}{n^k}
\left\{
\begin{array}{c}
k \\ r_k
\end{array}
\right\}
& (X \le r_k \le n) 
\\
\dfrac {n!}{n^k}
\left\{
\begin{array}{c}
k \\ X 
\end{array}
\right\}
&
(r_k \le X \le n)
\end{cases}
\label{star}
\end{equation}
By the argument 
in Case 1, in any cases provided 
$(X-1 \le )A-1 \le b \le n$, 
we have
\beq
\dfrac {n!}{n^k}
\left\{
\begin{array}{c}
k \\ b 
\end{array}
\right\}
\le
(Const.)
\frac {1}{
n^{\alpha}
p^{\log n + \alpha}
}.
\eeq
Hence 
\beq
1 
& \le &
\frac {n!}{n^k}
\sum_{b = A-1}^n
p^{n-b}
\left\{
\begin{array}{c}
k \\ b 
\end{array}
\right\}
\le
(Const.)
\frac {p^M}{
n^{\alpha}
p^{\log n + \alpha}
}.
\eeq
Therefore
$
n^{\alpha}
p^{\log n + \alpha}
\le
(Const.)
p^M
$
from which we have
$M \ge (Const.) \log n$.
%
%

%

\section{Appendix}
\subsection{Some elementary facts on the shuffle algebra}
We 
collect two facts on the shuffle algebra which are used in this paper. 
\begin{lemma}
\label{grouping}
Suppose 
$k \le n-1$
and 
we rewrite 
 ${\mathbf B}_k=\sum_{\alpha\in G_{k,p}}\alpha\shuffle W_{k,n}$
 by grouping the terms by the head (or top) letter as follows.
 \[
 {\mathbf B}_k 
 = 
 \sum_{t\in (C_p\times[k]) \cup\{(0,k+1)\}} 
 {\mathbf C}_{k}(t),
 \]
 where 
 ${\mathbf C}_k (t)$ 
 is the sum of the elements in 
 ${\mathbf B}_k$ 
 whose leading letter is 
 $t$. 
 We then have
 \[
 {\mathbf C}_k (t){\mathbf B}_1=
 \begin{cases}
  {\mathbf B}_k & t\in[p]\times[k],\\
  {\mathbf B}_{k+1} & t=(0,k+1).
 \end{cases}
 %
 \]
\end{lemma}
\begin{proof}
%
We regard each 
$\alpha \in G_{k, p}$
as a word, so that we denote by 
$i(\alpha)$
the leading letter, and by 
$\widetilde{\alpha}$
the remaining ones : 
$\alpha = i(\alpha) \widetilde{\alpha}$.
Then we write 
\beq
{\bf B}_k
&=&
\sum_{\alpha \in G_{k, p}}
(
i(\alpha) \widetilde{\alpha}
)
\shuffle
W_{k, n}
\\
&=&
\sum_{\alpha \in G_{k, p}}
i(\alpha) 
\Bigl(
\widetilde{\alpha}
\shuffle
W_{k, n}
\Bigr)
+
\sum_{\alpha \in G_{k, p}}
(0, k+1)
\Bigl(
\alpha \shuffle W_{k+1,n}
\Bigr)
\\
&=&
\sum_{t\in (C_p\times[k]) \cup\{(0,k+1)\}} 
{\mathbf C}_{k}(t).
\eeq
We note 
that the expression 
$i(\alpha) 
\Bigl(
\widetilde{\alpha}
\shuffle
W_{k, n}
\Bigr)$
stands for the concatenation  ; e.g., 
$1 (23 + 32) = 123 + 132$.

(i)
$t \in [p] \times [k]$ : 
%
\beq
C_k (t)
&=&
\sum_{
\substack{
\alpha \in G_{k, p}
\\
t = i(\alpha)
}
}
i(\alpha) 
\Bigl(
\widetilde{\alpha}
\shuffle
W_{k, n}
\Bigr)
=
\sum_{
\substack{
\alpha \in G_{k, p}
\\
t = i(\alpha)
}
}
t
\Bigl(
\widetilde{\alpha}
\shuffle
W_{k, n}
\Bigr).
\eeq
For 
$t \in [p] \times [k]$
and 
$q \in [p]$, 
let 
$t_q$
be the 
$q$-shift of colors in
$t$ : 
$t = (s,i) 
\mapsto
t_q = (s+q, i)$.
Applying 
${\bf B}_1$
from the right is equivalent to shifting colors of the first alphabet and then inserting it randomly.
Since 
the shuffle operator is associative, we have 
\beq
C_k (t) {\bf B}_1
&=&
\sum_{
\substack{
\alpha \in G_{k, p}
\\
t = i(\alpha)
\\
q \in [p]
}
}
t_q \shuffle 
\left(
\widetilde{\alpha} \shuffle W_k
\right)
=
\sum_{
\substack{
\alpha \in G_{k, p}
\\
t = i(\alpha)
\\
q \in [p]
}
}
\left(
t_q \shuffle 
\widetilde{\alpha} 
\right)
\shuffle W_k
=
\left(
\sum_{
\substack{
\alpha \in G_{k, p}
\\
t = i(\alpha)
\\
q \in [p]
}
}
t_q \shuffle 
\widetilde{\alpha} 
\right)
\shuffle W_k.
\eeq
Equation 
$
\displaystyle\sum_{
\substack{
\alpha \in G_{k, p}
\\
t = i(\alpha)
\\
q \in [p]
}
}
t_q \shuffle 
\widetilde{\alpha} 
=
\sum_{\beta \in G_{k, p}}
\beta
$
yields the result.\\
(ii)
$t=(0, k+1)$ : 
we can argue similarly as in (i). 
\QED
\end{proof}
The lemma below 
is an elementary fact in the linear algebra, which yields the eigenvalues and the corresponding eigenspaces of a matrix. 
\begin{lemma}
\label{eigenspace}
Suppose that 
$A, E_1, \cdots, E_m$
are nonzero 
$n \times n$
matrices satisfying 
\beq
A^k
&=&
\lambda_1^k E_1 + \cdots + \lambda_m^k E_m, 
\quad
k =0, 1, \cdots, 
\\
\lambda_i &\ne& \lambda_j, 
\quad
i \ne j.
\eeq
Then 
$P(x) = 
\prod_{j=1}^m (x - \lambda_j)$
%
%
is the minimal polynomial of 
$A$  
and 
\beq
E_i E_j &=& \delta_{ij} E_i, 
\quad
i, j = 1, 2, \cdots, m
\\
A E_i &=& \lambda_i E_i. 
\eeq
\end{lemma}
\begin{remark}
Since
$x \in Ran\, E_i$
satisfies 
$Ax = A E_i x = \lambda_i E_i x$, 
and since letting 
$k = 0$
in the assumption implies 
$I = E_1 + \cdots + E_m$, 
$\lambda_1, \cdots, \lambda_m$
are the eigenvalues of 
$A $ 
with
$Ran \, E_1, \cdots, Ran \, E_m$
being the corresponding eigenspaces.
\end{remark}
\begin{proof}
Let 
$Q$
be a polynomial.
Then by assumption
\begin{equation}
Q(A)
=
Q(\lambda_1)E_1 + \cdots + Q(\lambda_m) E_m.
\label{Q}
\end{equation}
Since 
$P(\lambda_i) = 0$, 
$i = 1, 2, \cdots, m$, 
we have
%
$
P(A)
=
P(\lambda_1) E_1 + \cdots + P(\lambda_m) E_m = 0.
$
%
On the other hand, let 
\beq
P_s (x)
&:=&
\prod_{j \ne s}
\frac {x - \lambda_j}{\lambda_s - \lambda_j}
=
(Const.)
\frac {P(x)}{x - \lambda_s}, 
\quad
s = 1, 2, \cdots, m.
\eeq
Then 
%
$
P_s(\lambda_i)
=
\left\{
\begin{array}{cc}
0 & (i \ne s) \\
1 & (i = s)
\end{array}
\right.
$
%
so that
%
$
P_s(A)
=
P_s(\lambda_1) E_1 + \cdots + P_s(\lambda_m) E_m
=
E_s
(\ne 0)
$
%
and hence 
$P$
is the minimal polynomial of 
$A$.
Plugging
$Q (\lambda) = P_i (\lambda) P_j (\lambda)$
and 
$Q (\lambda) = \lambda P_i (\lambda)$
respectively in 
(\ref{Q}), 
we have
\beq
E_i E_j
&=&
P_i(A) P_j(A)
%
=
(P_i P_j)(A)
\\
&=&
P_i(\lambda_1) P_j(\lambda_1) E_1
+ \cdots + 
P_i(\lambda_m) P_j(\lambda_m) E_m
=
\delta_{i, j}
E_i 
\\
A E_i
&=&
Q(A)
=
\lambda_1 P_i (\lambda_1) E_1 
+ \cdots +
\lambda_m P_i(\lambda_m) E_m
=
\lambda_i E_i.
\eeq
\QED
\end{proof}
%

\subsection{Asymptotics of Stirling numbers}
The following lemma 
is well-known, but we provide a proof for completeness.
%
\begin{lemma}
\label{lem:limit}
Let 
$\lambda > 0$.
Then if 
$n$
and 
$k$
goes to infinity satisfying 
$n e^{- \frac kn} \to \lambda$, 
we have 
\beq
{ k \brace n }
\frac {n!}{n^k}
\to
e^{- \lambda}.
\eeq
\end{lemma}
%
\begin{proof}
We consider putting 
$k$
balls uniformly at random into 
$n$
boxes.
Then 
$p (k,n):=
{ k \brace n }
\frac {n!}{n^k}$
is equal to the probability that no boxes are empty.
We aim to show 
$p (k,n) \to e^{- \lambda}$.
By the inclusion-exclusion principle,
\begin{equation}
p (k,n)
=
\sum_{j=0}^n
(-1)^{j}
{ n \choose j }
\left(
1 - \frac jn
\right)^k.
\label{**}
\end{equation}
Here we use the following estimates : 
\beq
&&
{ n \choose j }
=
\frac {
n (n-1) \cdots (n -j +  1)
}
{
j!
}
\begin{cases}
\le \dfrac {n^j}{j!}
\\
=
\dfrac {n^j}{j!}
\left(
1 - {\cal O}
\left(
\dfrac {j^2}{n}
\right)
\right)
\end{cases}
\\
&&
0 \le x \le 1/2
\quad
\Longrightarrow
\quad
e^{-x - x^2} \le 1 - x \le e^{-x}.
\eeq
Then we have 
\begin{equation}
\frac {1}{j!}
\left(
n e^{ - \frac kn }
\right)^j
e^{- k \left( \frac jn \right)^2 }
\left(
1 - {\cal O} \left( \frac {j^2}{n} \right)
\right)
\le 
{ n \choose j }
\left(
1 - \frac jn 
\right)^k
\le
\frac {1}{j!}
\left(
n e^{ - \frac kn }
\right)^j.
\label{***}
%
\end{equation}
Now 
we use the fact 
$k = n \log n - n \cdot \log (\lambda + o(1))$
and the estimate 
(\ref{***})
to apply the dominated convergence theorem on 
(\ref{**}), 
yielding the conclusion. 
\QED
\end{proof}
We 
next turn to the general case. 
%
%
\begin{lemma}
\label{lem:bound}
Suppose 
\begin{equation}
\frac {n}{\sqrt{k}} \to \infty, 
\quad
\frac kn - \log \sqrt{k} \to \infty.
\label{assump}
\end{equation}
(1)
Then for any 
$\delta > 0$, 
the following bound is valid for sufficiently large 
$n$. 
\beq
{ k \brace n }
& \le &
\frac {n^k}{n!}
\exp
\left[
- n e^{- \frac kn }
\right]
e^{ 
\frac 12 
e^{ - \frac kn}
}
\left(
1 + o(1)
\right).
\eeq
(2)
In particular, when 
$k = n \log n - c_n n$, 
$c_n \ll \log n$,
we have 
\beq
{ k \brace n }
& \le &
\frac {n^k}{n!}
\exp
\left[
- n e^{- \frac kn }
\right]
\left(
1 + 
{\cal O}
\left(
\frac {\log n}{n^{1 - \epsilon}}
\right)
\right), 
\quad
\forall \epsilon > 0.
\eeq
\end{lemma}
%
\begin{proof}
Lemma \ref{lem:bound}
follows directly from the result by Menon \cite{Menon} which is stated here as Lemma \ref{Menon} below.
In fact, we have
\beq
{ k \brace n}
&=&
\frac {n^k}{n!}
\exp
\left[
- n 
e^{- \frac kn}
\cdot
e^{ \frac {D}{n} }
\right]
\left(
1 + R
\right), 
\quad
D := \frac 12 
\left(
1 - \frac {1}{6n}
\right).
\eeq
$R$
is defined in the statement of Lemma \ref{Menon}.\\
(1)
Under the assumption 
(\ref{assump}), one has 
$R = o(1)$.
Then 
it suffices to use the inequality 
$e^{D/n} \ge 1 + D/n$
and noting that 
$D \le 1/2$.
\\
(2)
If 
$k = n \log n - c n$, 
$e^{ 
\frac 12 e^{ - k/n }
}
=
e^{ 
\frac {e^{c_n}}{2n}
}
=
1 + 
{\cal O} 
\left(
\dfrac {\log n}{n}
\right)$
and 
the error term in Lemma \ref{Menon} satisfies 
%
%
%
$R = {\cal O}
\left(
\dfrac {\log n}{n^{1 - \epsilon}}
\right)$
for any 
$\epsilon > 0$.
\QED
\end{proof}

{\bf Remark}
{\it 
We 
use Lemma \ref{lem:bound} several times in the proof of Theorem \ref{thm:cutoff}, so that we shall check the assumption 
(\ref{assump})
is valid in all cases. }\\

(0)
$k = n \log n - c n$, 
$c \ll \log n$ : 
\beq
\frac {n}{\sqrt{k}}
&=&
\frac {n}
{
\sqrt{
n \log n 
\left(
1 - \frac {c}{\log n}
\right)
}
}
=
\sqrt{
\frac {n}
{
\log n 
\left(
1 - \frac {c}{\log n}
\right)
}
}
\to \infty
\\
\frac kn - \log \sqrt{k}
&=&
(\log n - c)
-
\frac 12 \log 
\left[
n \log n 
\left(
1 - \frac {c}{\log n}
\right)
\right]
\\
&=&
(\log n - c)
-
\frac 12 \log n
-
\frac 12 
\log \log n
-
\frac 12
\log 
\left(
1 - \frac {c}{\log n}
\right)
\to \infty
\eeq

(1)
$n = r_k = \dfrac {k}{\log k}(1 + o(1))$ : 
\beq
\frac {r_k}{\sqrt{k}}
&=&
\frac {
\dfrac {k}{\log k}(1 + o(1))
}
{
\sqrt{k}
}
=
\frac {\sqrt{k}}{\log k}
(1+o(1))
\to
\infty
\\
\frac {k}{r_k} - \log \sqrt{k}
&=&
\frac {k}{
\dfrac {k}{\log k}(1 + o(1))
}
-
\frac 12 \log k
=
\frac 12 \log k (1+o(1))
\to
\infty
\eeq
(2)
$n$
is replaced by 
$n - \log n = n(1 + o(1))$ : 
\beq
\frac {n}{\sqrt{k}}
&=&
\frac { 
n (1+o(1))
}
{
\sqrt{ n \log n }(1 + o(1))
}
=
\sqrt{
\frac {n}{\log n}
}
(1+o(1))
\to \infty
\\
\frac kn -\log \sqrt{k}
&=&
\frac {n \log n (1+o(1))}{n}
-
\frac 12 \log n + o(1)
=
\frac 12 \log n + o(1)
\to \infty.
\eeq
%
%

%
\begin{lemma}
{\bf (\cite{Menon}, Theorem 2.2)}
\label{Menon}
Suppose 
\beq
\frac {n}{\sqrt{k}} \to \infty, 
\quad
\frac kn - \log \sqrt{k} \to \infty.
\eeq
Then 
\beq
{ k \brace n }
&=&
\frac {n^k}{n!}
\exp
\left[
- e^{\lambda}
\right]
\left(
1 + R
\right)
\\
\mbox{ where }
\quad
\lambda
&:=&
\log n - \frac kn + \frac {1}{2n} - \frac {1}{12 n^2}. 
\\
R
&:=&
1 + 
\frac 1n
e^{- k/n}
\left(
\frac {k+n}{2}
-
\frac {k^2}{8 n^2}
+
\frac {k}{3n}
+
\frac {1}{24}
\right)
\\
&&
-
e^{- 2k/n}
\left(
\frac {k+n}{2}
-
\frac {7 k^2}{8 n^2}
-
\frac {k}{4n}
-
\frac 18
\right)
-
n e^{- 3 k/n}
\left(
\frac {3 k^2}{4n^2}
+
\frac {7k}{6n}
+
\frac {7}{12}
\right)
\\
&&
+
e^{- 4 k/n}
\frac {(k+n)^2}{8}
+ 
R_1
\\
R_1
&:=&
{\cal O}
\left(
\frac {k^2}{n^4}
e^{- k/n}
+
\frac {k^2}{n^3}
e^{- 2k/n}
+
\frac {k^2}{n^2}
e^{- 3k/n}
+
\frac {k^2}{n}
e^{- 4k/n}
+
k^3 e^{- 4 k/n}
\right).
\eeq
\end{lemma}
%

\vspace*{1em}
\noindent {\bf Acknowledgement }
This work is partially supported by 
JSPS KAKENHI Grant 
Number 20K03659(F.N.).
\\
%


{\bf Ethics declarations }
The authors claim no conflict of interests. 
\\

{\bf Date Availability Statement }
Data sharing 
is not applicable to this article as no datasets were generated or analyzed during the current study.

\bibliographystyle{plain}
\bibliography{top2randcoloredperm}

\end{document}